\documentclass[a4paper,11pt,oneside,reqno]{amsart}
\usepackage{amsmath,amsfonts,amsthm,amssymb,dsfont,tikz-cd}
\usepackage{geometry,enumitem,hyperref}
\title{Raising to powers on the unit circle}
\author{Yilong Zhang}
\address{Mathematisches Institut, Universität Bonn, Endenicher Allee 60, D-53115 Bonn, Germany}
\email{s58yzhan@uni-bonn.de}
\date{\today}

\newenvironment{claim}[1]{\par\noindent\textbf{Claim.}\space#1}{}
\newenvironment{pfcl}[1]{\par\noindent\emph{Proof.}\space#1}{}

\theoremstyle{plain}
\newtheorem{theorem}{Theorem}[section]
\newtheorem{lemma}[theorem]{Lemma}
\newtheorem{proposition}[theorem]{Proposition}
\newtheorem{fact}[theorem]{Fact}
\newtheorem{corollary}[theorem]{Corollary}
\newtheorem{conjecture}[theorem]{Conjecture}
\newtheorem{maintheorem}{Theorem}

\theoremstyle{definition}
\newtheorem{definition}[theorem]{Definition}

\DeclareMathOperator{\cl}{cl}
\DeclareMathOperator{\Cl}{Cl}
\DeclareMathOperator{\Th}{Th}
\DeclareMathOperator{\ex}{ex}
\DeclareMathOperator{\kex}{kex}
\DeclareMathOperator{\acl}{acl}
\DeclareMathOperator{\dcl}{dcl}
\DeclareMathOperator{\trd}{trd}
\DeclareMathOperator{\mr}{MR}
\DeclareMathOperator{\ld}{ld}
\DeclareMathOperator{\md}{md}
\DeclareMathOperator{\rcf}{RCF}
\DeclareMathOperator{\rc}{rc}
\DeclareMathOperator{\ac}{ac}
\DeclareMathOperator{\pc}{pc}
\DeclareMathOperator{\st}{st}
\DeclareMathOperator{\fin}{fin}
\DeclareMathOperator{\rich}{rich}
\DeclareMathOperator{\tp}{tp}

\DeclareMathOperator{\rk}{rk}

\DeclareMathOperator{\zar}{Zar}

\DeclareMathOperator{\oring}{oring}
\def\indep{\mathrel{\raise0.2ex\hbox{\ooalign{\hidewidth$\vert$\hidewidth\cr\raise-0.9ex\hbox{$\smile$}}}}}

\newcommand{\A}{\mathcal{A}}
\newcommand{\B}{\mathcal{B}}
\newcommand{\C}{\mathcal{C}}

\newcommand{\F}{\mathcal{F}}
\newcommand{\I}{\mathcal{I}}

\newcommand{\M}{\mathcal{M}}

\newcommand{\R}{\mathcal{R}}
\newcommand{\U}{\mathcal{U}}
\newcommand{\W}{\mathcal{W}}
\newcommand{\X}{\mathcal{X}}
\newcommand{\Y}{\mathcal{Y}}
\newcommand{\Z}{\mathcal{Z}}
\newcommand{\bC}{\mathbb{C}}
\newcommand{\bN}{\mathbb{N}}
\newcommand{\bQ}{\mathbb{Q}}
\newcommand{\bR}{\mathbb{R}}
\newcommand{\bZ}{\mathbb{Z}}
\newcommand{\ba}{\boldsymbol{a}}
\newcommand{\bb}{\boldsymbol{b}}
\newcommand{\bc}{\boldsymbol{c}}
\newcommand{\bd}{\boldsymbol{d}}

\newcommand{\br}{\boldsymbol{r}}
\newcommand{\bs}{\boldsymbol{s}}
\newcommand{\bt}{\boldsymbol{t}}
\newcommand{\bu}{\boldsymbol{u}}
\newcommand{\bv}{\boldsymbol{v}}
\newcommand{\bw}{\boldsymbol{w}}
\newcommand{\bx}{\boldsymbol{x}}
\newcommand{\by}{\boldsymbol{y}}
\newcommand{\bz}{\boldsymbol{z}}

\newcommand{\Cf}{\mathcal{C}^{\fin}}
\newcommand{\KVS}{K\operatorname{-VS}}
\newcommand{\ag}[1]{{\langle#1\rangle}}
\newcommand{\se}{\subseteq}

\begin{document}

\begin{abstract}
  We study the expansion of the real field by the graphs of power functions on the unit circle. Under a natural number-theoretic conjecture, we prove that adding such dense subsets does not increase the topological complexity of definable sets: every open definable set remains semialgebraic. The proof uses a two-sorted structure that separates the linear and algebraic data, inspired by Zilber's raising to powers. Using Hrushovski's amalgamation method, we construct and axiomatize a class of rich structures, and then show that the intended structure is a model. This provides a new example of a tame expansion of the real field by dense trajectories.
\end{abstract}

\maketitle

\section{Introduction}

A central theme in tame geometry is understanding the definable sets in expansions of the real field $\bR_{\oring}=(\bR,<,+,\cdot,0,1)$. By Tarski's theorem, the definable sets in $\bR_{\oring}$ are precisely the semialgebraic sets. A natural question then arises: what happens when we expand the real field by adding new, analytically natural functions or relations? This paper investigates this question for expansions by trajectories on the torus, specifically by the graphs of the power functions on the unit circle.

For a real number $r$, consider the trajectory
\[\Gamma_r:=\{(e^{2\pi i t},e^{2\pi i r t})\in S^1\times S^1\se\bR^4\mid t\in \bR\}.\]
These sets are well-studied geometric objects. When $r$ is irrational, $\Gamma_r$ is a dense subset of the torus. The structure we study is the expansion of the real field by this entire family of trajectories, i.e., $(\bR_{\oring},(\Gamma_r)_{r\in\bR})$.

Our work fits into the broader program, pioneered by Chris Miller, of classifying expansions of the real field by trajectories of linear vector fields. In \cite{Mi11}, Miller showed that an expansion by locally closed trajectories is either d-minimal or defines the set of integers. However, the situation of expansions by dense trajectories was unclear. In contrast to Miller's result, our setting presents a starkly different yet well-behaved case. The central result of this paper is that, under a natural number-theoretic conjecture, adding all these dense trajectories introduces no new topological complexity: every open definable set remains semialgebraic. This establishes a new example of a non-trivial tame expansion by dense sets, complementing the existing classification.

More precisely, the essential property we use is the following fact observed by Kronecker \cite{Kr84}:

\begin{fact}\label{fact:Kronecker}
  Let $a_1,...,a_n\in\bR$ such that $(1,a_1,...,a_n)$ is $\bQ$-linearly independent. Let $p$ be the following map:
  \begin{align*}
    \bR\,&\longrightarrow\;(\bR/\bZ)^n\\
    x\,&\longmapsto([a_1x],...,[a_nx])
  \end{align*}
  Then $p(\bN)$ is dense in $(\bR/\bZ)^n$.
\end{fact}

For each $\ba=(a_1,...,a_n)\in\bR^n$, let $\Gamma_{\ba}^n$ be the following trajectory on the $n$-torus:
\[\Gamma_{\ba}^n:=\{(e^{2\pi ia_1t},...,e^{2\pi ia_nt})\in(S^1)^n\mid t\in \bR\}.\]
We will show that every open definable set in $(\bR_{\oring},(\Gamma_{\ba}^n)_{n\in\bN,\ba\in\bR^n})$ is semialgebraic.

To obtain this result, we reframe the problem using a two-sorted structure that separates the algebraic and linear data. For a subfield $K\se\bR$ extending $\bQ$, we define
\[\bR^K:=(\bR_{\KVS},\bR_{\oring},\exp),\]
where the first sort is $\bR$ viewed as a $K$-vector space, the second sort is the real ordered field, and $\exp:\bR_{\KVS}\to S^1\se(\bR_{\oring})^2$ is given by $\exp(t)=e^{2\pi it}$. This structure, inspired by Zilber's work on raising to powers \cite{Zi03}, is a powerful setting because it defines all $(\Gamma_r)_{r\in K}$ and allows us to apply tools from both geometric stability and tame geometry. We prove the following main theorem assuming the Schanuel conjecture for $K$-powers (SC$_K$), a weaker form of the Schanuel conjecture (SC). By Bays, Kirby, and Wilkie \cite{Ki10}, SC$_K$ holds for a generic $K$.

\begin{maintheorem}
  Assume SC$_K$ holds for some finitely generated $K\se\bR$. Then every open subset of $(\bR_{\oring})^n$ definable in $\bR^K$ is semialgebraic.
\end{maintheorem}

Since every definable set is defined using finitely many symbols, we obtain the desired result. Note that SC implies SC$_K$ for every subfield $K$ of $\bC$.

\begin{corollary}
  Assume SC$_K$ holds for every finitely generated $K\se\bR$. Then every open subset of $(\bR_{\oring})^n$ definable in $\bR^\bR$ is semialgebraic. In particular, every open definable set in $(\bR_{\oring},(\Gamma_{\ba}^n)_{n\in\bN,\ba\in\bR^n})$ is semialgebraic.
\end{corollary}

Our proof proceeds by constructing and axiomatizing a class of models for the potential theory of $\bR^K$. The main technical tool is the Hrushovski construction, where a predimension function $\delta$ governs the interaction between the $K$-linear structure and the algebraic structure. The Schanuel-type conjecture SC$_K$ precisely states that this predimension function is bounded below, ensuring our intended structure $\bR^K$ can be seen as a rich structure in an elementary class.

More precisely, we apply the Hrushovski construction to the following setting:
\begin{itemize}
  \item a $K$-vector space $V$,
  \item a real closed field $R$,
  \item its unit circle $S=\{(x,y)\in R^2\mid x^2+y^2=1\}$ as a subgroup of $(R^2)^\times$,
  \item a surjective group homomorphism $\ex:V\to S$ whose kernel is a $\bZ$-group,
  \item the class $\C$ of all such structures $(V,R,\ex)$, and
  \item the predimension function $\delta(\bx):=\ld^K(\bx)+\trd(\ex(\bx))-\ld^\bQ(\bx)$,
\end{itemize}
where $\ld^K$ denotes the $K$-linear dimension and $\trd$ the transcendence degree.

\emph{Strong embeddings} between structures in $\C$ are embeddings with non-negative predimension. Fix a small structure $\I\in\C$ and let $\C_\I$ consist of all structures in $\C$ strongly extending $\I$. We provide a first-order axiomatization $T_\I$ of the class $\C_\I$, incorporating a fact about atypical intersections, a weaker form of the Conjecture on Intersections with Tori (CIT).

We call a structure $\U\in\C_\I$ \emph{rich} if for every strong embedding $\X\to\Y$ between small structures, every strong embedding $\X\to\U$ extends to a strong embedding $\Y\to\U$. We show that $\C_\I$ is an amalgamation class, and hence contains a rich structure. The axiomatization of richness consists of two parts: an axiom scheme KR that captures the Kronecker-type behaviour of the trajectories, and an axiom scheme EC that captures existential closedness of rich structures.

\begin{maintheorem}\label{thm:axiomatization}
  Rich structures in $\C_\I$ are axiomatized by $T_\I\cup\operatorname{KR}\cup\operatorname{EC}$.
\end{maintheorem}

We prove that the intended structure $\bR^K$ always satisfies $\operatorname{KR}$ and $\operatorname{EC}$. The property EC is shown by analyzing unlikely intersections using the Ax-Schanuel theorem.

The final step uses a criterion of Boxall and Hieronymi \cite{Bo12}, which allows us to transfer the equivalence of types of generic points into the definability of open sets, yielding the main statement of open definable sets for rich structures in $\C_\I$. Combining this with the fact that $\bR^K$ satisfies the theory of rich structures, we obtain the main theorem.

This paper is structured as follows. In Section \ref{sec:pre}, we recall some facts in model theory and arithmetic geometry, mainly focusing on the Hrushovski construction and the exponential function. Section \ref{sec:con} is devoted to the setup and the basic properties of the construction. In Section \ref{sec:C_I}, we find the axioms stating that the predimension function is non-negative, thus axiomatizing the class $\C_\I$. In Section \ref{sec:kro}, we study the axioms for the Kronecker property. Section \ref{sec:str} is devoted to strong pairs, objects that characterize strong embeddings. Section \ref{sec:axi} is devoted to the complete axiomatization of rich structures. In Section \ref{sec:ope}, we examine the open definable sets in a rich structure. In Section \ref{sec:mod}, we show that $\bR^K$ is a model of the theory of rich structures.

\subsection*{Acknowledgements}

I would like to thank my advisor, Prof. Philipp Hieronymi, for suggesting this topic and for his continuous guidance and support. I am grateful to the Mathematical Institute of the University of Bonn for its support during this research. The author was partially supported by the Deutsche Forschungsgemeinschaft (DFG, German Research Foundation) under Germany's Excellence Strategy -- EXC-2047/2 -- 390685813.

\subsection*{Notations}

We use curly letter $\X,\Y$ for structures and $E^\X$ for the interpretation of $E$ in $\X$. We use bold letters $\ba,\bb$ for finite tuples, and we write $\ba\bb$ for a concatenation of two tuples. Let $\ld^K,\md,\trd$ denote the $K$-linear dimension for some field $K$, the multiplicative degree, and the transcendence degree respectively.

Let $S$ be a set. For a subset $W\se S^{n+m}$ and a tuple $\ba\in S^m$, we write $W(-,\ba)$ or $W_{\ba}$ for the \emph{fiber} of $\ba$, i.e.,
\[W(-,\ba)=W_{\ba}=\{\,\bx\in S^n\mid(\bx,\ba)\in W\,\}.\]

\section{Preliminaries}\label{sec:pre}

\subsection{Geometric structures}

In this subsection, we review some basic facts of geometric structures. We focus on strongly minimal and o-minimal structures. Standard references are \cite{Ma02} and \cite{Dr98}.

\subsubsection{Strongly minimal structures}

\begin{definition}
  A structure $\A$ is \emph{minimal} if every definable subset of $A$ is finite or cofinite. A theory $T$ is \emph{strongly minimal} if every model of $T$ is minimal.
\end{definition}

Strongly minimal theories form the most fundamental class of theories, including vector spaces and algebraically closed fields.

\begin{fact}
  The algebraic closure operator $\acl$ in a strongly minimal structure is a pregeometry.
\end{fact}

Let $\A$ be a sufficiently saturated strongly minimal structure, and let $d$ denote the dimension function induced by its algebraic closure.

\begin{fact}
  Let $\ba\se A$ and $B\se A$. Then $\mr(\ba/B)=d(\ba/B)$.
\end{fact}

\begin{fact}
  Suppose $X\se A^n$ is a definable set over a small set $B\se A$. Then
  \[\mr(X)=\max_{\bx\in X}\mr(\bx/B).\]
\end{fact}

A \emph{generic point} of $X$ over $B$ is an element of $X$ achieving the maximum Morley rank. We write $\dim(X)$ for the Morley rank of $X$. In particular, if $\A=(A,+,\cdot,0,1)$ is an algebraically closed field, then $d$ is the transcendence degree, and the Morley rank of an algebraic set equals to its Krull dimension.

\subsubsection{o-minimal structures}

\begin{definition}
  A structure $\R=(R,<,...)$ where $<$ is a linear order is \emph{o-minimal} if every definable subset of $R$ is a finite union of intervals and points.
\end{definition}

Examples of o-minimal structures include real closed fields and the expansion of the real field by restricted analytic functions $\bR_{\operatorname{an}}$.

\begin{fact}
  The definable closure operator $\dcl$ in an o-minimal structure is a pregeometry.
\end{fact}

\begin{definition}
  Suppose $\R$ is an o-minimal structure, and $X\se\R^n$ is a definable set. The \emph{topological dimension} of $X$, denoted by $\dim(X)$, is the maximum $k\leq n$ such that the image $\pi(X)$ of some coordinate projection $\pi:R^n\to R^k$ has non-empty interior.
\end{definition}

Let $\R$ be a sufficiently saturated o-minimal structure, and let $d$ denote the dimension function induced by its definable closure.

\begin{fact}
  Suppose $X\se R^n$ is a definable set over a small set $B\se R$. Then
  \[\dim(X)=\max_{\bx\in X}d(\bx/B).\]
  A \emph{generic point} of $X$ over $B$ is an element of $X$ achieving the maximum $\dcl$-dimension.
\end{fact}

The above fact also holds for the real ordered field $\bR_{\oring}=(\bR,<,+,\cdot,0,1)$:

\begin{fact}
  Suppose $X\se\bR^n$ is a semialgebraic set defined over a countable set $B\se\bR$. Then
  \[\dim(X)=\max_{\bx\in X}\,\trd(\bx/B).\]
\end{fact}

\subsubsection{Vector spaces}

Let $K$ be a field and let $(V,+,0,(\lambda_k)_{k\in K})$ be a $K$-vector space, where $\lambda_k$ is the scalar multiplication by $k$. The algebraic closure operator in $V$ is the $K$-linear span, and $\emptyset$-definable sets are given by Boolean combinations of the following sets:

\begin{definition}
  A subspace $D$ of $V^n$ is \emph{$K$-linear} if $D=\ker(M_D:V^n\to V^m)$ for some $M_D\in K^{m\times n}$ of rank $m$. We define $\dim D:=n-m$, which equals the Morley rank of $D$.
\end{definition}

Note that $D$ is definably isomorphic to $V^{\dim D}$ and hence has Morley degree $1$.

\begin{fact}
  Let $V_0\se V$ be a $K$-subspace, and $\bx\in V^n$. We have
  \begin{align*}
    \ld^K(\bx)&=\min\{\dim D\mid\bx\in D\text{ for some }K\text{-linear }D\},\\
    \ld^K(\bx/V_0)&=\min\{\dim D\mid\bx\in D+\by\text{ for some }K\text{-linear }D\text{ and some }\by\se V_0\}.
  \end{align*}
\end{fact}

\begin{lemma}
  Let $A,B\se V^n$ be $K$-linear subspaces. Then $A+B$ is also a $K$-linear subspace.
\end{lemma}
\begin{proof}
  Let $R(M_A)$ and $R(M_B)$ be the row spaces of $M_A$ and $M_B$ respectively. Then $A+B=\ker(R(M_A))+\ker(R(M_B))=\ker(R(M_A)\cap R(M_B))$, which is a $K$-linear subspace defined by a basis of $R(M_A)\cap R(M_B)$.
\end{proof}

\begin{corollary}\label{co:maxQ-linear}
  Let $K'$ be a subfield of $K$. Then every $K$-linear subspace $D\se V^n$ contains a unique maximal $K'$-linear subspace $D_{K'}$.
\end{corollary}

\subsection{Hrushovski construction}

In this subsection, we state the concepts and facts of Hrushovski's predimension construction method.

\subsubsection{Predimension}

Let $A$ be a set and let $\cl$ be a modular pregeometry on $A$. Consider a function $\delta:\mathcal{P}_{\fin}(A)\to\bZ$, where $\mathcal{P}_{\fin}(A)$ consists of all finite tuples of $A$.

\begin{definition}
  A function $\delta$ is \emph{preserved by $\cl$} if for all finite tuples $\bx,\by\se A$ with $\cl(\bx)=\cl(\by)$ we have $\delta(\bx)=\delta(\by)$.
\end{definition}

For a function $\delta$ preserved by $\cl$ we may extend its definition to subsets of finite $\cl$-dimension by setting $\delta(\cl(\bx)):=\delta(\bx)$.

\begin{definition}
  A function $\delta$ preserved by $\cl$ is a \emph{predimension function} if for every finite dimensional $\cl$-closed $X$ and $Y$, the submodularity condition holds:
  \[\delta(X\cup Y)+\delta(X\cap Y)\leq\delta(X)+\delta(Y).\]
\end{definition}

Let $\delta$ be a predimension function preserved by $\cl$.

\begin{definition}
  \label{de:loc}
  The \emph{localization of $\delta$ at $\by$} is the function $\delta(\bx/\by):=\delta(\bx\by)-\delta(\by)$.
\end{definition}

\begin{definition}
  Let $Y$ be an arbitrary $\cl$-closed subset of $A$. The \emph{localization of $\delta$ at $Y$} is the function $\delta_Y:\mathcal{P}_{\fin}(A)\to\bZ\cup\{-\infty\}$ given by
  \[\delta(\bx/Y):=\inf\{\delta(\bx/Y')\mid\text{finite dimensional cl-closed }Y'\text{ satisfying }\cl(\bx)\cap Y\se Y'\se Y\}.\]
\end{definition}

\begin{definition}
  \label{def:ss}
  Let $Y\se X$ be $\cl$-closed subsets of $A$. We say $Y$ is \emph{strong} in $X$ if $\delta(\bx/Y)\geq 0$ for every $\bx\se X$, denoted by $Y\leq X$.
\end{definition}

We collect some basic properties of strong subsets.

\begin{fact}\label{prop:str}
  Let $X,Y,Z$ be $\cl$-closed subsets of $A$.
  \begin{enumerate}
    \item (Subset) If $X\se Y\se Z$ and $X\leq Z$, then $X\leq Y$.
    \item (Transitivity) If $X\leq Y$ and $Y\leq Z$, then $X\leq Z$.
    \item (Union) Let $((X_i)_{i\in I},\leq)$ be a directed set. Then $X_j\leq\bigcup_{i\in I}X_i$ for every $j\in I$.
    \item (Intersection) If $X\leq Z$ and $Y\leq Z$, then $X\cap Y\leq Z$.
  \end{enumerate}
\end{fact}

\begin{definition}
  Let $Y\leq X$. We say $X$ is \emph{prealgebraic} over $Y$ if for every $\by\se Y$ there is $\bz\se Y$ such that $\delta(\by\bz/X)=0$, denoted by $Y\leq_0 X$.
\end{definition}

\begin{fact}\label{le:prealg}
  Let $X\se Y\se Z$ be $\cl$-closed subsets of $A$.
  \begin{enumerate}
    \item If $X\leq_0 Y$ and $Y\leq_0 Z$, then $X\leq_0 Z$.
    \item Suppose $X\leq_0 Y$. Then $X\leq Z$ if and only if $Y\leq Z$.
  \end{enumerate}
\end{fact}

\begin{definition}
  Let $X\se A$. The \emph{hull} of $X$ in $A$, denoted by $[X]_A$, is the smallest $\cl$-closed strong subset of $A$ containing $X$. We omit the subscript $A$ when it is clear from the context.
\end{definition}

\begin{fact}
  \label{le:min}
  Suppose $\delta:\mathcal{P}_{\fin}(A)\to\bZ$ is bounded below. Then $[X]$ exists for every $X\se A$. Furthermore, every $\bx\se A$ can be extended to a finite tuple $\bx\by$ such that $[\bx]=\cl(\bx\by)$.
\end{fact}

\begin{definition}
  The \emph{dimension function associated to $\delta$} is $d:\mathcal{P}_{\fin}(A)\to\bN$ given by $d(\bx)=\delta([\bx])$.
\end{definition}

Let $\Cl$ denote the pregeometry induced by $d$. Clearly,
\[\Cl(X)=\bigcup\{\by\se A\mid\delta(\by/[X])=0\}.\]
In other words, $\Cl(X)$ is the largest prealgebraic set over $[X]$.

\subsubsection{Amalgamation}

Let $L$ be a (possibly multi-sorted) language and let $\C$ be a class of $L$-structures. Fix a relation symbol $\Delta\in L$. Assume $\C$ is closed under images, i.e., for every $\A,\B\in\C$ if $\A$ is embedded into $\B$ then the image of $\A$ is a structure in $\C$.

Let $\cl$ and $\delta$ be a modular pregeometry and a predimension function defined uniformly on $\Delta^\C$, i.e., for each $\X\in\C$ we have a modular pregeometry $\cl^\X$ on $\Delta^\X$ and a predimension function $\delta^\X$ on $\Delta^\X$ preserved by $\cl^\X$, such that for every embedding $i:\X_1\to\X_2$ between structures in $\C$ and every $\bx\se\Delta^{\X_1}$ we have $i(\cl^{\X_1}(\bx))=\cl^{\X_2}(i(\bx))$ and $\delta^{\X_1}(\bx)=\delta^{\X_2}(i(\bx))$.

\begin{definition}
  \label{de:se}
  Let $\X,\Y$ be structures in $\C$.

  If $\X$ is a substructure of $\Y$ and $\Delta^\X$ is strong in $\Delta^\Y$, then we say $\Y$ is a \emph{strong extension} of $\X$, denoted by $\X\leq\Y$.

  An embedding $\iota:\X\to\Y$ is a \emph{strong embedding} if $\iota(\X)\leq\Y$.
\end{definition}

\begin{definition}
  Let $\X\in\C$. We say $\X$ is \emph{finitely $\leq$-generated} if for every chain
  \[\X_0\leq\X_1\leq\X_2\leq...\]
  of structures in $\C$ with $\X=\bigcup_{i<\omega}\X_i$ there is $k<\omega$ such that $\X_k=\X$.
\end{definition}

Let $\Cf$ be the subclass of $\C$ consisting of all finitely $\leq$-generated structures.

\begin{definition}
  A structure $\U\in\C$ is \emph{rich} if for every $\X,\Y\in\Cf$ and strong embeddings $f:\X\to\Y,\,g:\X\to\U$, there is a strong embedding $h:\Y\to\U$ such that $h\circ f=g$.
  \begin{center}\begin{tikzcd}
    \X \arrow[rd, "g"] \arrow[d, "f"']\\
    \Y \arrow[r, "h", dashed] & \U
  \end{tikzcd}\end{center}
\end{definition}

\begin{definition}
  $\C$ has the \emph{amalgamation property for strong embeddings} (\emph{APS}) if for all structures $\X,\Y,\Z\in\C$ and strong embeddings $f_1:\X\to\Y$ and $f_2:\X\to\Z$, there exist $\W\in\C$ and strong embeddings $g_1:\Y\to\W$, $g_2:\Z\to\W$, such that $g_1\circ f_1=g_2\circ f_2$.
  \begin{center}\begin{tikzcd}
    \X \arrow[d, "f_1"'] \arrow[r, "f_2"] & \Z \arrow[d, "g_2", dashed]\\
    \Y \arrow[r, "g_1", dashed] & \W
  \end{tikzcd}\end{center}
\end{definition}

\begin{definition}\label{def:amalgamationClass}
  $\C$ is an \emph{amalgamation class} if it is non-empty and satisfies the following conditions:
  \begin{enumerate}
    \item The cardinality of structures in $\Cf$ is bounded by some cardinal $\kappa$.
    \item $\C$ is closed under the union of every $\leq$-chain.
    \item $\C$ has the amalgamation property for strong embeddings.
  \end{enumerate}
\end{definition}

\begin{fact}
  \label{le:rs}
  Suppose that $\C$ is an amalgamation class. Then $\C$ contains a rich structure.
\end{fact}

To show a class has APS, we often construct free amalgams of structures.

\begin{definition}
  A theory $T$ has the \emph{free amalgamation property} if for every structures $\X,\Y,\Z\vDash T$ and embeddings $f_1:\X\to\Y$, $f_2:\X\to\Z$, there exist $\W\vDash T$ and embeddings $g_1:\Y\to\W$, $g_2:\Z\to\W$, such that
  \begin{itemize}
    \item $g_1\circ f_1=g_2\circ f_2$, and
    \item $g_1(Y)$ and $g_2(Z)$ are free over $g_1\circ f_1(X)$, i.e., $g_1(Y)\indep^{\operatorname{acl}}_{g_1\circ f_1(X)}g_2(Z)$.
  \end{itemize}
\end{definition}

Strongly minimal theories and o-minimal theories all have the free amalgamation property \cite[Lemma 1.2]{Pi88}.

\subsubsection{Rich structures}

This part is devoted to useful properties of rich structures. We are in the same setting as the previous section. Further assume that $\C$ is an amalgamation class so that it contains a rich structure.

\begin{definition}
  We say $\C$ is \emph{prime} if $\C$ contains a structure $\I$ such that for every $\X\in\C$ we have $\I\leq\X$. Such $\I$ is called the \emph{initial structure}.
\end{definition}

\begin{definition}
  Let $\I\in\Cf$. The \emph{localization} of $\C$ at $\I$ is the subclass
  \[\C_\I:=\{\X\in\C\mid\I\leq\X\}\]
  associated with
  \begin{itemize}
    \item language $L(\I)$,
    \item pregeometry $\cl_\I$ given by the localization of $\cl$ at $\Delta^\I$, and
    \item predimension function $\delta_\I$ given by the localization of $\delta$ at $\Delta^\I$.
  \end{itemize}
\end{definition}

\begin{definition}
  We say $\C$ is \emph{abundant} if for every $\X\leq\A\in\C$ and $a\in\A$ the following hold:
  \begin{itemize}
    \item If $\X\in\Cf$ then there is $\Y\in\Cf$ such that $a\in\Y$ and $\X\leq\Y\leq\A$.
  \end{itemize}
\end{definition}

For the rest of this subsection, suppose $\C$ is prime and abundant.

For $\A_1,\A_2\in\C$, let $\F(\A_1,\A_2)$ be the following set of partial isomorphisms:
\[\F(\A_1,\A_2)=\{\,f:\X_1\xrightarrow{\sim}\X_2\mid\X_i\leq\A_i,\,\X_i\in\Cf\,\}.\]
\begin{fact}
  \label{le:bf}
  Let $\U_1,\U_2$ be rich structures in $\C$. Then $\F(\U_1,\U_2)$ forms a back-and-forth system.
\end{fact}
\begin{corollary}\label{le:elemEq}
  Let $\U_1,\U_2$ be rich structures in $\C$. If $f:\X_1\xrightarrow{\sim}\X_2$ is a partial isomorphism where $\X_1\leq\U_1,\X_2\leq\U_2$, then $f$ is elementary. In particular, $\U_1$ and $\U_2$ are elementarily equivalent.
\end{corollary}

\begin{definition}
  A structure $\A\in\C$ is \emph{$\leq$-existentially closed} if given a quantifier-free $L(\A)$-formula $\phi(\bx)$, a strong extension $\B$ of $\A$, and a tuple $\bb$ in $\B$ such that $\B\vDash\phi(\bb)$, we can find $\ba$ in $\A$ such that $\A\vDash\phi(\ba)$.
\end{definition}

\begin{fact}
  Let $\U$ be a rich structure in $\C$. Then $\U$ is $\leq$-existentially closed.
\end{fact}

\subsection{Torus and exponential}

Let $K$ be an algebraically closed field of characteristic $0$. Consider the algebraic $n$-torus $(K^\times)^n$. Every torus has an algebraic group structure induced by the multiplicative group. The following fact gives a characterization of their algebraic subgroups.

For every $(k\times n)$-integer matrix $M=(m_{ij})$, define a map from the $n$-torus to the $k$-torus:
\begin{align*}
  (-)^M:\;\;(K^\times)^n\;\;&\longrightarrow\;\;(K^\times)^k\\
  (x_1,...,x_n)&\longmapsto\left(\prod_{i=1}^n x_i^{m_{1i}},...,\prod_{i=1}^n x_i^{m_{ki}}\right).
\end{align*}

\begin{fact}[{\cite[Section 3.2]{Bo06}}]
  Every algebraic subgroup $A$ of the $n$-torus $(K^\times)^n$ is defined by a system of equations written as follows:
  \[\bx^M=\mathds{1},\]
  where $\bx$ is an $n$-tuple of variables and $M$ is a $(k\times n)$-integer matrix where $k\leq n$. Moreover, $M$ can be taken to have full rank, and then the dimension of $A$ equals $n-k$.
\end{fact}

For an algebraic subgroup $A\se(K^\times)^n$, we write $M_A$ for a full-rank integer matrix such that $A$ is defined by equations $M_A(\bx)=\mathds{1}$.

The Conjecture on Intersections with Tori is a special case of what is now called the Zilber--Pink conjecture. It predicts that atypical components of intersections are controlled by finitely many proper algebraic subgroups. It was proposed by Zilber in 2002 \cite[Conjecture 1]{Zi02}.

\begin{conjecture}[CIT]
  \label{con:CIT}
  Given an algebraic variety $W\se K^n$ defined over $\bQ$, there is a finite collection $\mu(W)$ of proper algebraic subgroups of $(K^\times)^n$, satisfying the following property:

  If $S$ is an atypical component of the intersection of $W$ and a proper algebraic subgroup $A\se(K^\times)^n$, i.e.,
  \[\dim S>\dim W+\dim A-n,\]
  then $S$ is contained in some $B\in\mu(W)$.
\end{conjecture}

For our purpose, the following weaker form is sufficient.

\begin{fact}[Weak CIT, {\cite[Theorem 4.6]{Ki09}}]
  \label{fact:WCIT}
  Given a constructible family $(W_{\bc})_{\bc\in P}$ in $K^n$, i.e., constructible $W\se K^{n+m}$ and $P\se K^m$, there is a finite set $\mu(W)$ of proper algebraic subgroups of $(K^\times)^n$ satisfying the following property:

  For every $\bc\in P$ and every proper algebraic subgroup $A\se(K^\times)^n$, if $S$ is an atypical component of the intersection of $W_{\bc}$ and a coset $\alpha A$, then there exist $B\in\mu(W)$ and a constant $\beta$ such that $S$ is contained in $\beta B$. Moreover, $S$ is typical with respect to $\beta B$.
\end{fact}

The weak CIT is a direct consequence of the so-called uniform Schanuel property, which is stated in Kirby's paper \cite{Ki09} in the context of differential fields. The following fact is its special case for the complex exponential $\exp:\bC^n\to(\bC^\times)^n$.

\begin{fact}[Uniform Ax-Schanuel, {\cite[Theorem 4.3]{Ki09}}]
  Given a constructible family $(V_{\bc})_{\bc\in P}$ in $\bC^n\times (\bC^\times)^n$, there is a finite set $\mu(V)$ of proper algebraic subgroups of $(\bC^\times)^n$ satisfying the following property:

  For every $\bc\in P(\bC)$ and irreducible analytic variety $U\se\Gamma_{\exp}\cap V_{\bc}$, if $\dim V_{\bc}-\dim U<n$, then there exist $B\in\mu(V)$ and a constant $\beta$ such that $p(U)\se\beta B$, where $p$ is the projection $\bC^n\times (\bC^\times)^n\to(\bC^\times)^n$.
\end{fact}

The following is the Ax-Schanuel theorem in its original form. It can be deduced from the above uniform version. We say functions $f_1,...,f_n\in\bC[[x_1,...,x_m]]$ are \emph{$\bQ$-linearly dependent modulo $\bC$} if some of their $\bQ$-linear combination is in $\bC$.

\begin{fact}[Ax-Schanuel, \cite{Ax71}]
  \label{fact:AxSch}
  Let $f_1,...,f_n\in\bC[[x_1,...,x_m]]$ be $\bQ$-linearly independent modulo $\bC$. Then
  \[\trd(f_1,...,f_n,e^{f_1},...,e^{f_n}/\bC)\geq n+\rk J(f_1,...,f_n),\]
  where $J(f_1,...,f_n)=\left(\frac{\partial f_i}{\partial x_j}\right)$.
\end{fact}

The Mordell--Lang property plays an important role in solving finite rank subgroups of an algebraic group. We are interested in the case of characteristic zero, which is exactly the following theorem proved by Laurent \cite{La83}.

\begin{fact}
  \label{fact:ml}
  Let $\Gamma$ be a finite rank subgroup of $(K^\times)^n$ and let $W$ be a subvariety of $(K^\times)^n$. Then there exist a natural number $r$, elements $\gamma_1,...,\gamma_r$ of $\Gamma$, and algebraic subgroups $A_1,...,A_r$ of $(K^\times)^n$, such that $\gamma_i A_i\se W$ and
  \[W\cap\Gamma=\bigcup_{i=1}^r\gamma_i(A_i\cap\Gamma).\]
\end{fact}

The famous Schanuel conjecture concerns about the algebraic relation of the complex exponential.

\begin{conjecture}[SC]
  \label{con:SC}
  For all $\bx\se\bC$, we have $\trd(\bx,e^{\bx})\geq\ld^\bQ(\bx)$.
\end{conjecture}

In order to obtain a lower bound of the predimension function, we need the following weaker form of the Schanuel conjecture, called the Schanuel conjecture for $K$-powers.

\begin{conjecture}[SC$_K$]
  \label{con:SCK}
  Let $K$ be a subfield of $\bC$ of finite transcendence degree. Then for all $\bx\se\bC$, we have
  \[\ld^K(\bx)+\trd(e^{\bx})-\ld^\bQ(\bx)\geq -\trd(K).\]
\end{conjecture}

By Bays, Kirby, and Wilkie \cite[Theorem 1.3]{Ki10}, SC$_K$ holds in the following special case.

\begin{fact}
  Let $K=\bQ(\bb)$ where $\bb\se\bC$ is an exponentially algebraically independent tuple. Then for all $\bx\se\bC$, we have
  \[\ld^K(\bx)+\trd(e^{\bx})-\ld^\bQ(\bx)\geq 0.\]
\end{fact}

\section{Construction}\label{sec:con}

Let $K$ be a finitely generated subfield of $\bR$ properly containing $\bQ$.

\begin{definition}
  A \emph{$K$-powered field} is a tuple $(V,R,\ex)$ containing:
  \begin{itemize}
    \item a $K$-vector space $V$,
    \item a real closed field $R$,
    \item its induced algebraically closed field $R^2$,
    \item its unit circle $S=\{(x,y)\in R^2\mid x^2+y^2=1\}$ as a subgroup of $(R^2)^\times$,
    \item a surjective group homomorphism $\ex:V\to S$ with kernel $\kex$ a $\bZ$-group, i.e., an elementary extension of $(\bZ,+,0,1)$.
  \end{itemize}
  Let $\C$ be the class of all $K$-powered fields.
\end{definition}

Recall that stereographic projection gives a definable homeomorphism:
\begin{align*}
  p_{\st}:S\setminus\{(1,0)\}&\longrightarrow\;R\\
  (x,y)\;&\longmapsto\frac{y}{1-x}.
\end{align*}
Let $\ag{X}^K$ denote the $K$-vector space generated by $X\se V$, let $\ag{X}^{\rc}$ denote the real closure generated by $X\se R$, let $\ag{X}^{\ac}$ denote the algebraic closure generated by $X\se R^2$, and let $\ag{X}^S:=\ag{X}^{\ac}\cap S$. For every $X\se S$ we have $\ag{X}^{\ac}\cap R=\ag{p_{\st}(X)}^{\rc}$, so we will write $\ag{X}^{\rc}$ with no ambiguity.

\begin{definition}
  Let $\A$ be a $K$-powered field, and let $X\se V^\A$. The $K$-powered field generated by $X$ in $\A$, denoted by $\ag{X}^{\pc}_\A$, is the smallest substructure of $\A$ containing $X\cup\kex^\A$ that is also a $K$-powered field.
\end{definition}

\begin{lemma}\label{le:pc-closure}
  $\B=\ag{X}^{\pc}_\A$ is obtained by iterating $\ag{\cdot}^K$ and $\ex^{-1}(\ag{\ex(\cdot)}^S)$.
\end{lemma}
\begin{proof}
  Define $(V_n,R_n)$ inductively:
  \begin{itemize}
    \item $V_1:=\ag{X\cup\kex^\A}^K$.
    \item $R_n:=\ag{\ex(V_n)}^{\rc}$.
    \item $S_n:=\ag{\ex(V_n)}^S=(R_n)^2\cap S$.
    \item $V_{n+1}:=\ag{\ex^{-1}(S_n)}^K$.
  \end{itemize}
  Note that in order to make $\ex$ surjective, $V_{n+1}$ must contain some (and hence all) preimage of $S_n$. Take $V^\B=\bigcup_{i<\omega}V_i$ and $R^\B=\bigcup_{i<\omega}R_i$. Verify that $\ex^\B$ is surjective.
\end{proof}

As a consequence, for $\A\se\B\in\C$ with $\kex^\A=\kex^\B$ and $X\se V^\A$, the $K$-powered field generated by $X$ in $\A$ is the same as in $\B$. We may omit the subscript if it is clear from the context.

To prepare for the Hrushovski construction, we take $\Delta$ to be the sort $V$. Let $\cl$ be the $\bQ$-linear span on $V$ and define the following predimension function:
\[\delta(\bx):=\ld^K(\bx)+\trd(\ex(\bx))-\ld^\bQ(\bx),\]
for finite tuple $\bx\se V$. It is straightforward to verify that $\delta$ is a predimension function preserved by $\cl$.

Observe that $\ld^\bQ(\bx/\kex)=\md(\ex(\bx))$ and localizating $\delta$ is the same as localizing each component, i.e.,
\[\delta(\by/X)=\ld^K(\by/X)+\trd(\ex(\by)/\ex(X))-\ld^\bQ(\by/X).\]

The following lemma implies that, to test the strongness of an extension of a $K$-powered field, we only need to test its generating set.

\begin{lemma}
  Let $\A\in\C$, and let $X$ be a strong $\bQ$-subspace of $V^\A$ such that $\ag{X}^{\pc}_\A=\A$. Then $V^\A$ is prealgebraic over $X$.
\end{lemma}
\begin{proof}
  In the proof of Lemma \ref{le:pc-closure}, we have the following prealgebraic chain:
  \[X\leq_0 V_1\leq_0\ex^{-1}(S_1)\leq_0 V_2\leq_0\ex^{-1}(S_2)\leq_0...,\]
  because either $\ld^K$ or $\trd$ vanishes. Fact \ref{le:prealg} yields $X\leq_0 V^\A$.
\end{proof}

Combining with Fact \ref{le:prealg}, we have the following.

\begin{corollary}\label{cor:strongStructure}
  Let $\A\se\B\in\C$, and let $X\leq V^\A$ such that $\ag{X}^{\pc}_\A=\A$. Then $\A\leq\B$ if and only if $X\leq V^\B$.
\end{corollary}

Given a covering map from a $\bQ$-vector space to a unit circle, we can always extend it to a $K$-powered field, without changing the kernel.

\begin{proposition}\label{prop:extendToStructure}
  Let $V_0'$ be a $K$-vector space, let $V_0$ be a $\bQ$-subspace of $V_0'$, let $R_0$ be a real closed field, let $S_0$ be its unit circle, and let $\ex_0:V_0\to S_0$ be a group homomorphism such that $\ag{V_0}^K=V_0'$ and $\ag{\ex_0(V_0)}^S=S_0$, and $\ker(\ex_0)$ is a $\bZ$-group. Then there is a structure $\A=(V^\A,R^\A,\ex^\A)\in\C$ extending $(V_0',R_0,\ex_0)$ such that $V_0$ is strong in $V^\A$, $\ker(\ex^\A)=\ker(\ex_0)$, and $\ag{V_0}^{\pc}_\A=\A$.
\end{proposition}
\begin{proof}
  First note that since $\ker(\ex_0)$ is a $\bZ$-group, we have $\operatorname{Tor}(S_0)\se\ex_0(V_0)$. Therefore, the map $\ex^\A$ will be surjective as long as its image contains a set whose divisible hull is $S^\A$.
  
  Inductively construct $R_n,S_n,V_n,V_n'$ and $\ex_n:V_n\to S_n,\ex_n':V_n'\to S_{n+1}$ such that
  \begin{itemize}
    \item $S_n$ is the unit circle of $R_n^2$,
    \item $\ag{V_n}^K=V_n'$ and $\ag{\ex_n(V_n)}^S=S_n$,
    \item $\ex_n'$ extends $\ex_n$ and $\ex_{n+1}$ extends $\ex_n'$,
    \item $V_n$ is strong in $V_n'$ and $V_n'$ is strong in $V_{n+1}$,
    \item $\ker(\ex_n)=\ker(\ex_n')=\ker(\ex_0)$.
  \end{itemize}
  \begin{enumerate}
    \item We first extend $\ex_0$ to a homomorphism $\ex_0':V_0'\to S_1$ for some extension $S_1\se R_1^2$ of $S_0\se R_0^2$. We want $V_0$ to be strong, i.e.,
    \[\delta(\bx/V_0)=\ld^K(\bx/V_0)+\trd(\ex(\bx)/\ex(V_0))-\ld^\bQ(\bx/V_0)\geq 0,\]
    for all $\bx\se V_0'$. Take a $\bQ$-basis $\{a_i\mid i\in I_0\}$ of $V_0'$ over $V_0$. Since $\ag{V_0}^K=V_0'$, the image $\{\ex(a_i)\}_{i\in I_0}$ must be algebraically independent over $\ag{\ex(V_0)}^S=S_0$. There is a type $p((y_i)_{i\in I_0})$ stating that $(y_i)_{i\in I_0}$ is contained in $S$ and is algebraically independent over $S_0$. Let $R_1$ be the real closed field generated by $R_0$ and a realization of $p$, and let $\ex_0'$ send $a_i$'s to the realization. Verify $\ker(\ex_0)=\ker(\ex_0')$.
    \item Suppose we already have $\ex_{n-1}':V_{n-1}'\to S_n$. We want a surjective homomorphism $\ex_n:V_n\to S_n$ extending $\ex_{n-1}'$, where $V_n$ is a $\bQ$-subspace of some $K$-vector space $V_n'$ satisfying $\ag{V_n}^K=V_n'$. Take a multiplicative basis $\{b_i\mid i\in J_n\}$ of $S_n$ over $\ex(V_{n-1}')$. Since $\ag{\ex(V_{n-1}')}^S=S_n$, if some $a_i$'s are mapped to $b_i$'s, then $a_i$'s must be $K$-independent over $V_{n-1}'$, (o.w., $V_{n-1}'$ cannot be strong). There is a type $q((x_i)_{i\in J_n})$ stating that $(x_i)_{i\in J_n}$ are $K$-independent over $V_{n-1}'$. Let $V_n'$ be the $K$-vector space generated by $V_{n-1}'$ and a realization $(a_i)_{i\in J_n}$ of $q$, let $V_n$ be the $\bQ$-subspace generated by $V_{n-1}'$ and $(a_i)_{i\in J_n}$, and let $\ex_n$ send $a_i$'s to $b_i$'s. Verify $\ker(\ex_{n-1}')=\ker(\ex_n)$.
    \item Go back to step 1 substituting $(V_0,R_0,\ex_0)$ with $(V_n,R_n,\ex_n)$.
  \end{enumerate}
  Take $V^\A=\bigcup_{n<\omega}V_n,\,R^\A=\bigcup_{n<\omega}R_n,\,\ex^\A=\bigcup_{n<\omega}\ex_n$.
\end{proof}

Plugging in $V_1=K,R_0=\bQ^{\rc}$, and
\[\ex_0:V_0=\bQ\to\bQ/\bZ\to\operatorname{Tor}(S),\]
we obtain a naive $K$-powered field.

\begin{corollary}
  The class $\C$ is non-empty.
\end{corollary}

Similarly, given a strong set in a $K$-powered field, we can always find a strong $K$-powered subfield containing it preserving the kernel.

\begin{lemma}\label{le:minimalStrong}
  Let $\A\in\C$, and let $X$ be a strong $\bQ$-subspace of $V^\A$ containing $\bZ$. There is a $K$-powered subfield $\B\leq\A$ such that $\B=\ag{X}^{\pc}_\B$ and $\kex^\B=X\cap\kex^\A$.
\end{lemma}
\begin{proof}
  We construct $S_n,V_n,V_n'$ similar to the previous lemma. Let $V_0=X$.
  \begin{enumerate}
    \item Let $V_n'=\ag{V_n}^K$. Since $V_n$ is strong, we know $V_n'$ is strong and $V_n'\cap\kex^\A=V_n\cap\kex^\A$.
    \item Let $S_n=\ag{\ex(V_{n-1}')}^S$. Take a multiplicative basis $\{b_i\mid i\in I_n\}$ of $S_n$ over $\ex(V_{n-1}')$. For each $i$, take an element $a_i$ in $\ex^{-1}(b_i)$. Let $V_n=\ag{(a_i)_{i\in I_n}}^\bQ$. Verify that $V_n$ is strong and $V_n\cap\kex^\A=V_{n-1}'\cap\kex^\A$.
  \end{enumerate}
  Take $V^\B=\bigcup_{n<\omega}V_n,\,S^\B=\bigcup_{n<\omega}S_n,\,R^\B=\ag{S^\B}^{\rc}$.
\end{proof}

\subsection{Localization}

Fix a $K$-powered field $\I$ with $\kex^\I=\bZ$ generated by a finite-dimensional strong $\bQ$-subspace $X$. We may take $X$ to be $\bQ$-independent from $\kex^\I$. Let $\bb$ be a $\bQ$-basis of $X$, i.e., $X=\ag{\bb}^\bQ$. Consider the localization of $\C$ at $\I$:
\[\C_\I=\{\A\in\C\mid\I\leq\A\}.\]
Let $\C_\I'$ consist of all $K$-powered fields extending $\I$, i.e., $\C_\I'=\{\A\in\C\mid\I\se\A\}$. Clearly, $\C_\I'$ can be axiomatized by a set of axioms $T_\I'$. For each $\A\in\C_\I'$, Corollary \ref{cor:strongStructure} gives that $\I\leq\A$ if and only if $\ag{\bb}^\bQ\leq V^\A$, which is further characterized by the following lemma.

\begin{lemma}\label{le:equivC_I}
  Let $\A\in\C_\I'$. Then $\bb$ is $\bQ$-independent from $\kex^\A$, and we have $\I\leq\A$ if and only if the following conditions hold:
  \begin{enumerate}
    \item $\ld^K(\bx/\bb)=\ld^\bQ(\bx/\bb)=\ld^\bQ(\bx)$ for all $\bx\se\kex^\A$,
    \item $\delta(\bx/\bb\cup\kex^\A)\geq 0$ for all $\bx\se V^\A$.
  \end{enumerate}
\end{lemma}
\begin{proof}
  Indeed, $\ag{\bb}^\bQ\cap\kex^\A=\ag{\bb}^\bQ\cap\kex^\I=0$, i.e., $\bb$ is $\bQ$-independent from $\kex^\A$. Condition 1 is equivalent to $\ag{\bb}^\bQ\leq_0\ag{\bb\cup\kex^\A}^\bQ$, and Condition 2 is equivalent to $\ag{\bb\cup\kex^\A}^\bQ\leq V^\A$. The conclusion follows from Fact \ref{le:prealg}.
\end{proof}

\subsection{Finite generation and minimal extensions}

In this subsection, we define explicitly what it means for a $K$-powered field to be finitely generated. By examining minimal strong extensions in $\C_\I$, we will see that this definition coincides with the notion of finite $\leq$-generation.

\begin{definition}
  A $K$-powered field $\A$ is \emph{finitely generated} if $\ld^\bQ(\kex^\A)$ is finite and $\A=\ag{\bx}^{\pc}_\A$ for some finite $\bx\se V^\A$.
\end{definition}

\begin{proposition}\label{prop:minimalExtension}
  Let $\A\leq\B\in\C$ be a proper minimal strong extension. Then exactly one of the following holds:
  \begin{enumerate}
    \item (kernel) $\ld^\bQ(\kex^\B/V^\A)=1$ and $\B=\ag{V^\A}^{\pc}_\B$.
    \item (prealgebraic) $\kex^\A=\kex^\B$ and $\B=\ag{V^\A,\ba}^{\pc}_\B$ for some $\bQ$-independent tuple $\ba\se V^\B$ over $V^\A$ with $\delta(\ba/V^\A)=0$.
    \item (purely transcendental) $\kex^\A=\kex^\B$ and $\B=\ag{V^\A,x}^{\pc}_\B$ for every $x\in V^\B\setminus V^\A$.
  \end{enumerate}
\end{proposition}
\begin{proof}$\,$
  \begin{enumerate}
    \item Suppose $\kex^\A\neq\kex^\B$. Corollary \ref{cor:strongStructure} gives $\A\leq\ag{V^\A}^{\pc}_\B\leq\B$. Since $\A\leq\B$ is minimal, it must be that $\B=\ag{V^\A}^{\pc}_\B$. Take $b\in\kex^\B\setminus V^\A$. Notice that $X:=\ag{V^\A,b}^\bQ$ is strong in $V^\B$. By Lemma \ref{le:minimalStrong}, there is a $K$-powered subfield $\C\leq\B$ such that $\C=\ag{X}_\C^{\pc}$ and $\kex^\C=X\cap\kex^\B$. Minimality gives $\B=\C$ and hence $\kex^\B\se X$.
    \item Suppose $\kex^\A=\kex^\B$ and there is a $\bQ$-independent tuple $\ba\se V^\B$ over $V^\A$ with $\delta(\ba/V^\A)=0$. By Fact \ref{le:prealg}, we have $\A\leq\ag{V^\A,\ba}^{\pc}_\B\leq\B$. Minimality implies $\B=\ag{V^\A,\ba}^{\pc}_\B$.
    \item Suppose $\kex^\A=\kex^\B$ and every $\bQ$-independent tuple $\by\se V^\B$ over $V^\A$ gives $\delta(\by/V^\A)>0$. Let $x\in V^\B\setminus V^\A$. Clearly,
    \[\ld^K(x/V^\A)=\trd(\ex(x)/R^\A)=\ld^\bQ(x/V^\A)=1.\]
    Hence, $\delta(\by/\ag{V^\A,x}^\bQ)=\delta(\by,x/V^\A)-\delta(x/V^\A)\geq 0$ for every $\bQ$-independent tuple $\by\se V^\B$ over $V^\A$. In other words, $\ag{V^\A,x}^\bQ$ is strong in $V^\B$. Corollary \ref{cor:strongStructure} implies $\A\leq\ag{V^\A,x}^{\pc}_\B\leq\B$ and minimality implies $\B=\ag{V^\A,x}^{\pc}_\B$.\qedhere
  \end{enumerate}
\end{proof}

\begin{lemma}
  A $K$-powered field $\A\in\C_\I$ is finitely generated if and only if it is finitely $\leq$-generated.
\end{lemma}
\begin{proof}
  $(\Rightarrow)$ follows from the classification of minimal strong extensions.
  
  $(\Leftarrow)$ If there is no $\bx\se V^\A$ such that $\A=\ag{\bx}^{\pc}_\A$, then we can construct an infinite proper $\leq$-chain in $\A$ by applying Fact \ref{le:min} and Lemma \ref{le:minimalStrong} repeatedly. If $\ld^\bQ(\kex^\A)$ is infinite, then we can construct a chain with infinitely growing $\kex$.
\end{proof}

Recall that $\Cf_\I$ denotes the finitely $\leq$-generated structures in $\C_\I$.

\subsection{Rich structures}

In this subsection, we show that $\C_\I$ contains a rich structure, and establish basic properties of rich structures in $\C_\I$.

\begin{lemma}
  $\C_\I$ has the amalgamation property for strong embeddings.
\end{lemma}
\begin{proof}
  Given $\{\A_i=(V_i,R_i,\ex_i)\in\C_\I\mid i=0,1,2\}$ and strong embeddings $f_1:\A_0\to\A_1,\,f_2:\A_0\to\A_2$, we construct their free amalgam (not unique) $\A_1\otimes_{\A_0}\A_2$ as follows:

  Let $V_1\otimes_{V_0}V_2$ and $R_1\otimes_{R_0}R_2$ be free amalgams of $K$-vector spaces and real closed fields, and let $\ex:V_1+V_2\to S_1\cdot S_2$ be the group homomorphism extending $\ex_1:V_1\to S_1,\,\ex_2:V_2\to S_2$. To see that $\ker(\ex)$ is again a $\bZ$-group, one may verify the axioms in \cite{Na90}. Applying Proposition \ref{prop:extendToStructure} to $(V_1\otimes_{V_0}V_2,R_1\otimes_{R_0}R_2,\ex)$, we obtain a $K$-powered field $\A_3$ generated by $V_1+V_2$ as desired. To show $g_1(\A_1)\leq\A_3$ and $g_2(\A_2)\leq\A_3$, it suffices to show that $V_1$ and $V_2$ are strong in $V_1+V_2$. For every $\bx\se V_2$,
  \begin{align*}
    \delta(\bx/V_1)&=\ld^K(\bx/V_1)+\trd(\ex(\bx)/R_1)-\ld^\bQ(\bx/V_1)\\
    &=\ld^K(\bx/V_0)+\trd(\ex(\bx)/R_0)-\ld^\bQ(\bx/V_0)&&\text{by freenesss}\\
    &=\delta(\bx/V_0)\\
    &\geq 0.&&\qedhere
  \end{align*}
\end{proof}

\begin{proposition}
  $\C_\I$ contains a rich structure.
\end{proposition}
\begin{proof}
  Verify that $\C_\I$ is an amalgamation class (recall Definition \ref{def:amalgamationClass}): Every finitely generated $K$-powered field has cardinality $\leq\aleph_0$. Condition 2 is clear, and Condition 3 follows from the previous lemma.
\end{proof}

\begin{lemma}
  $\C_\I$ is prime and abundant. Therefore, rich structures in $\C_\I$ are back-and-forth equivalent and $\leq$-existentially closed.
\end{lemma}
\begin{proof}
  Clearly, $\I$ is an initial structure. To see that $\C_\I$ is abundant, let $\X\leq\A\in\C_\I$ with $\X=\ag{\bb}^{\pc}_\X\in\Cf_\I$ and let $a\in V^\A$. Apply Fact \ref{le:min} to find a tuple $\bc\se V^\A$ such that $\ag{a,\bb,\bc}^\bQ$ is strong in $V^\A$. By applying Lemma \ref{le:minimalStrong} to $\ag{a,\bb,\bc,\kex^\X}^\bQ$ we obtain a desired $\Y\in\Cf_\I$ with $\X\leq\Y\leq\A$ and $a\in V^\Y$.
\end{proof}

\section{Predimension inequality}\label{sec:C_I}

This subsection is devoted to first order characterization of $\delta$. One can find such a characterization for the complex case in Zilber's paper \cite{Zi15}, by which some of our argument is inspired. The goal is to express the two conditions in Lemma \ref{le:equivC_I}.

Let $\A=(V,R,\ex)\in\C_\I'$ be a $K$-powered field extending $\I$ and let $F=R^2$. Take a $K$-basis $\bb'$ of $\bb$, and denote $k=|\bb'|$. Clearly, Condition 1 is equivalent to $\ld^K(\bx/\bb')=\ld^\bQ(\bx/\bb')$ for all $\bx\se\kex^\A$.

For each $K$-linear subspace $D\se V^n$, let $\phi^{\kex}_D$ be the following sentence:
\[\forall\bx\in\kex^{n-k},\,\bx\bb'\in D\rightarrow\bx\bb'\in D_\bQ,\]
where $D_\bQ$ is defined in Corollary \ref{co:maxQ-linear}. Let $\Phi^{\kex}$ consist of all $\phi^{\kex}_D$ with $\dim(D)<n$.

\begin{lemma}
  For every $\A\in\C_\I'$,
  \[\A\vDash\Phi^{\kex}\text{ if and only if }\ld^K(\ba\bb')=\ld^\bQ(\ba\bb')\text{ for all }\ba\se\kex^\A.\]
\end{lemma}
\begin{proof}
  $(\Rightarrow)$ Suppose on the contrary that there is $\ba\in\kex^{n-k}$ such that $\ba\bb'$ is $\bQ$-independent but $K$-dependent. Then $\ba\bb'$ lies in some $K$-linear subspace $D$ of dimension $<n$. The sentence $\phi^{\kex}_D$ implies $\ba\bb'\in D_\bQ$, and hence $\ld^\bQ(\ba\bb')\leq\dim(D_\bQ)<n$. Contradiction.

  $(\Leftarrow)$ Let $\ba\in\kex^{n-k}$ such that $\ba\bb'\in D$ for some $K$-linear subspace $D$ of dimension $l<n$. We show $\ba\bb'\in D_\bQ$ by induction on $l$. The case $l=0$ is trivial. Take a $\bQ$-linear subspace $D'$ containing $\ba\bb'$ with $\dim(D')=\ld^\bQ(\ba\bb')=\ld^K(\ba\bb')\leq l$. If $D'\se D$ then we are done. If $D'\nsubseteq D$, then $\dim(D'\cap D)<l$, and we conclude by induction hypothesis.
\end{proof}

This finishes the first condition in Lemma \ref{le:equivC_I}. To express the second condition, we need to study intersections of varieties with algebraic subgroups of $(F^\times)^n$.

\begin{lemma}
  The map $ex^n:V^n\to S^n$ induces a bijection between $\bQ$-linear subspaces of $V^n$ and intersections of $S^n$ with irreducible algebraic subgroups of $(F^\times)^n$.
\end{lemma}
\begin{proof}
  Follows from the bijection between primitive subgroups of $\bZ^n$ and linear tori in $(F^\times)^n$.
\end{proof}

For each variety $W\se F^n$ over $\bQ(\ex(\bb))$, the weak CIT (Fact \ref{fact:WCIT}) gives a finite collection $\mu(W)=\{B_1,...,B_s\}$ of poper algebraic subgroups, where each $B_i$ is defined by a full-rank integer matrix $M_i$. Let $D\se V^n$ be a $K$-linear subspace defined by full-rank matrix $M_D\in K^{m\times n}$. By Corollary \ref{co:maxQ-linear}, for each $i$ there is $(M_iD)_\bQ$ defined by full-rank integer matrix $N_i$. By Fact \ref{fact:ml} applied to $(W^{M_i})^{N_i}$ and $\ex(\ag{\bb}^\bQ)$, we obtain $r_i$, $\{\gamma_{i,1},...,\gamma_{i,r_i}\}$, and $\{A_{i,1},...,A_{i,r_i}\}$. Let $\phi_{DW}$ be the following sentence:
\begin{align*}
  &\forall\bx\in V^n,\left(M_D\bx\se\ag{\bb\cup\kex}^K\bigwedge\ex(\bx)\in W\right)\\
  &\longrightarrow\bigvee_{i=1}^s\bigvee_{j=1}^{r_i}\left(\dim(A_{i,j})<\rk(N_i)\bigwedge\ex(N_iM_i\bx)\in\gamma_{i,j}A_{i,j}\right).
\end{align*}
Let $\Phi_\I$ consist of all $\phi_{DW}$ satisfying $\dim D+\dim W<n$.

\begin{lemma}
  For every $\A\in\C_\I'$,
  \[\A\vDash\Phi_\I\text{ if and only if }\delta(\ba/\bb\cup\kex)\geq 0\text{ for all }\ba\se V^\A.\]
\end{lemma}
\begin{proof}
  $(\Rightarrow)$ It suffices to show $\delta(\ba/\bb\cup\kex)\geq 0$ for all $\bQ$-independent $\ba\in V^n$ over $\bb\cup\kex$. Suppose on the contrary that there exists $\bQ$-independent $\ba\in V^n$ over $\bb\cup\kex$ with $\ld^K(\ba/\bb\cup\kex)+\trd(\ex(\ba)/\ex(\bb))<n$. Let $D\se V^n$ be a $K$-linear subspace such that $M_D\bx\se\ag{\bb\cup\kex}^K$ and $\ld^K(\ba/\bb\cup\kex)=\dim D$. Let $W\se F^n$ be the algebraic locus of $\ex(\ba)$ over $\bQ(\ex(\bb))$. Because
  \[\dim D+\dim W=\ld^K(\ba/\bb\cup\kex)+\trd(\ex(\ba)/\ex(\bb))<n,\]
  the formula $\phi_{DW}$ states that $\ex(\ba)$ must be multiplicatively dependent over $\ex(\bb)$. Contradiction.

  $(\Leftarrow)$ Let $D\se V^n$ be a $K$-linear subspace and $W\se F^n$ be a variety over $\bQ(\ex(\bb))$ with $\dim D+\dim W<n$. Suppose $M_D\ba\se\ag{\bb\cup\kex}^K$ and $\ex(\ba)\in W$. Then,
  \[\ld^\bQ(\ba/\bb\cup\kex)\leq\ld^K(\ba/\bb\cup\kex)+\trd(\ex(\ba)/\ex(\bb))\leq\dim D+\dim W<n.\]
  Take a coset $\alpha A$ of an algebraic subgroup $A$ such that $\ex(\ba)\in\alpha A$, $\alpha\se\ex(\ag{\bb}^\bQ)$, and $\dim(A)=\md(\ex(\ba)/\ex(\bb))$. Let $P$ be an irreducible component of $W\cap\alpha A$ containing $\ex(\ba)$. The following calculation indicates that $P$ is atypical:
  \begin{align*}
    \dim P&\geq\trd(\ex(\ba)/\ex(\bb))\\
    &\geq\md(\ex(\ba)/\ex(\bb))-\ld^K(\ba/\bb\cup\kex)\\
    &\geq\dim A-\dim D\\
    &>\dim A+\dim W-n.
  \end{align*}

  By weak CIT (Fact \ref{fact:WCIT}), there is $B_i\in\mu(V)$ and $\gamma$ such that $\bx^{M_i}=\gamma$ for all $\bx\in P$, and in particular, $\ex(M_i\ba)=\gamma$. Since $\gamma$ is defined by $W$ and $\alpha A$, it must be algebraic over $\bQ(\bb)$, i.e., $\trd(\gamma/\bb)=0$. As a consequence, $\delta(\gamma/\bb\cup\kex)=0$. Let $A',B_i'\se V^n$ denote the $\bQ$-linear subspaces defined by $M_A,M_i$ respectively. We have
  \[\dim(M_i A'\cap M_i D)\geq\ld^K(M_i\ba/\bb\cup\kex)=\md(\ex(M_i\ba)/\bb)=\dim A^{M_i}=\dim M_i A'.\]
  Therefore, $M_i A'\cap M_i D=M_i A'$, which is contained in $(M_i D)_\bQ$. Hence,
  \[(\ba^{M_i})^{N_i}\in(\alpha^{M_i})^{N_i}(A^{M_i})^{N_i}=(\alpha^{M_i})^{N_i}\se\ex(\ag{\bb}^\bQ).\]
  It suffices to show $\dim(W^{M_i})^{N_i}<\rk(N_i)$.
  
  Fact \ref{fact:WCIT} says that $P$ is typical with respect to $\ex(\ba)\cdot B_i$, i.e.,
  \begin{align*}
    \dim P&=\dim(W\cap\ex(\ba)\cdot B_i)+\dim(\ex(\ba)(A\cap B_i))-\dim(\ex(\ba)\cdot B_i)\\
    &=\dim W-\dim W^{M_i}+\dim(A'\cap B_i')-\dim B_i'.
  \end{align*}
  Also, $\delta(M_i\ba/\bb\cup\kex)=0$ implies $\delta(\ba/M_i\ba\cup\bb\cup\kex)=\delta(\ba/\bb\cup\kex)\geq 0$. Thus,
  \begin{align*}
    &\dim P+\dim(D\cap B_i')\\
    \geq&\trd(\ex(\ba)/\ex(\bb))+\ld^K(\ba/M_i\ba\cup\bb\cup\kex)\\
    \geq&\ld^\bQ(\ba/M_i\ba\cup\bb\cup\kex)\\
    =&\ld^\bQ(\ba/\bb\cup\kex)-\ld^\bQ(M_i\ba/\bb\cup\kex)\\
    =&\dim A'-\dim(M_i A')\\
    =&\dim(A'\cap B_i').
  \end{align*}
  Combining these two,
  \begin{align*}
    0&\leq\dim W-\dim W^{M_i}+\dim(D\cap B_i')-\dim B_i'\\
    &=\dim W+\dim D-\dim W^{M_i}-\dim M_i D-n+\rk M_i\\
    &<\rk N_i+\dim(M_iD)_\bQ-\dim W^{M_i}-\dim M_i D\\
    &\leq\rk N_i-\dim W^{M_i}.
  \end{align*}
  We conclude that
  \[\dim(W^{M_i})^{N_i}\leq\dim W^{M_i}<\rk N_i.\qedhere\]
\end{proof}

Let $T_\I$ be the conjunction of $T_\I'$, $\Phi^{\kex}$, and $\Phi_\I$.

\begin{corollary}
  The class $\C_\I$ is axiomatized by $T_\I$.
\end{corollary}

Following the same proof, we obtain the following proposition, which provides first-order characterization for purely transcendental strong extensions.

\begin{proposition}
  \label{prop:ptType}
  Let $\A$ be a $K$-powered field and let $\bc\se V^\A$. There is a partial type $\Psi(x)$ over $\bc$ such that for every $t\in V^\A$, we have that $\A\vDash\Psi(t)$ if and only if $\delta(\bx,t/\bc\cup\kex)>0$ for all $\bx\se V^\A$.
\end{proposition}
\begin{proof}
  The desired property is equivalent to the conjunction of the followings:
  \begin{itemize}
    \item $\ld^K(t/\bc\cup\kex)=\trd(\ex(t)/\ex(\bc))=\ld^\bQ(t/\bc\cup\kex)=1$,
    \item $\delta(\bx/t,\bc,\kex)\geq 0$ for all $\bx\se V^\A$.
  \end{itemize}
  The second can be stated by an axiom scheme similar to $\Phi_\I$ with $\bb$ replaced by $t\bc$.
\end{proof}

\section{Kronecker property}\label{sec:kro}

We isolate a set of axioms describing Kronecker's theorem -- Fact \ref{fact:Kronecker}. For all $r_1,...,r_n\in K$, let $f_{\br}$ be the following map:
\begin{align*}
  V&\longrightarrow\;S^n\\
  x&\longmapsto(\ex(r_1x),...,\ex(r_nx)).
\end{align*}

Let KR be the set of axioms stating that, for each $r_1,...,r_n\in K$ such that $(1,r_1,...,r_n)$ is $\bQ$-independent, we have $f_{\br}(\kex)$ is dense in $S^n$. Observe that KR ensures $f_{\br}(k\cdot\kex)$ being dense in $S^n$ for all such $\br$ and $k\in\bN_{>0}$.

Let $\A,\B$ be $K$-powered fields, let $V_0^\A,V_0^\B$ be $K$-subspaces of $V^\A,V^\B$ respectively, and let $\iota:V_0^\A\to V_0^\B$ be a $K$-linear isomorphism such that $\ex(V_0^\A)\equiv^{\rcf}\ex(V_0^\B)$, i.e., the map $\ex(V_0^\A)\to\ex(V_0^\B)$ induces an isomorphism as real closed fields. Let us identify $V_0^\A$ with $V_0^\B$ and simply write $V_0$. Assume $\A$ satisfies KR and is $|V_0|^+$-saturated.

\begin{lemma}\label{le:extendKernel}
  Take a $\bQ$-basis $(r_j)_{j\in J}$ of $K$ over $\bQ$. Let $y\in\kex^\B\setminus V_0$ such that $\ex(r_jy)_{j\in J}$ is algebraically independent over $\ex(V_0)$. Then there is $x\in\kex^\A\setminus V_0$ such that
  \[\ex\ag{x}^K\equiv^{\rcf}_{\ex(V_0)}\ex\ag{y}^K.\]
\end{lemma}
\begin{proof}
  Let $p(t)$ be the type stating that
  \begin{enumerate}
    \item $t\in V\setminus V_0$,
    \item $t\in y+k\cdot\kex$ for all $k\in\bN$,
    \item $\ex((r_jt)_{j\in J})\equiv_{\ex(V_0)}^{\rcf}\ex((r_jy)_{j\in J})$.
  \end{enumerate}
  Because $\ex(r_jy)_{j\in J}$ is algebraically independent over $\ex(V_0)$, every finite segment of $p(t)$ says $t$ lies in an open subset of $S^n$ for some $n$, which is realized in $\A$ by KR.

  Let $x\in\kex^\A$ be a realization of p. By 2, we have $\ex(\frac{x}{k})=\ex(\frac{y}{k})$ for all $k\in\bN$, and the desired isomorphism follows from 3.
\end{proof}

The following lemma states that partial isomorphisms of strong subsets can be extended to its $K$-span by adding a kernel element.

\begin{lemma}\label{le:addKernel}
  Let $\bx\in(V^\A)^n$, $\by\in(V^\B)^n$ be $K$-independent tuples over $V_0$ such that $\ag{\bx,V_0}^\bQ\leq V^\A$ and $\ag{\by,V_0}^\bQ\leq V^\B$. If the map $\bx\mapsto\by$ induces an isomorphism $\ex(\bx)\equiv^{\rcf}_{\ex(V_0)}\ex(\by)$, then there is $\bx':=\bx+\bz$ for some $\bz\se\kex^\A$ such that $\ag{\bx',V_0}^\bQ\leq V^\A$, $\bx'$ is $K$-independent over $V_0$, and $\ex\ag{\bx'}^K\equiv^{\rcf}_{\ex(V_0)}\ex\ag{\by}^K$.
\end{lemma}
\begin{proof}
  Let $h=(h_j:V^n\to V)_{j\in J}$ be a set of $K$-linear maps such that $\{\bx,h(\bx)\}$ is a $\bQ$-basis of $\ag{\bx}^K$. Strongness implies that $\ex(h_j(\bx))_{j\in J}$ is algebraically independent over $\ex(\bx\cup V_0)$, and the same for $\by$. Let $R'\vDash\rcf$ be an extension of $\ag{\ex\ag{\bx,V_0}^\bQ}^{\rc}$ containing $(s_j)_{j\in J}$ such that
  \[\ex(\bx),(s_j)_{j\in J}\equiv^{\rcf}_{\ex(V_0)}\ex(\by),\ex(h_j(\by))_{j\in J}.\]
  Let $p((t_i)_{i<n})$ be the type stating the follows:
  \begin{enumerate}
    \item $\bt\se\kex$,
    \item $\bt$ is $\bQ$-independent over $\ag{\bx,V_0}^\bQ$.
    \item $\ex(h_j(\bt))_{j\in J}\equiv_{\ex(\bx\cup V_0)}^{\rcf}\left(\frac{s_j}{\ex(h_j(\bx))}\right)_{j\in J}$.
  \end{enumerate}
  It is clear by KR that $p(\bt)$ is satisfiable in $\A$. Let $\bz\se\kex^\A$ be a realization of $p$. Thus,
  \[\ex(\bx',h(\bx'))=\ex(\bx,h(\bx+\bz))\equiv^{\rcf}_{\ex(V_0)}\ex(\by,h(\by)).\]
  Since $\ag{\bx,V_0}^\bQ\leq V^\A$, we have
  \[\ld^K(\bz/\bx,V_0)=\ld^\bQ(\bz/\bx,V_0)=n.\]
  In other words, $\bz$ is $K$-independent over $\bx\cup V_0$, and therefore, $\bx\bz$ is $K$-independent over $V_0$.
  Lastly, since $\ag{\bx,V_0}^\bQ\leq_0\ag{\bx\bz,V_0}^\bQ$ and $\ag{\bx',V_0}^\bQ\leq_0\ag{\bx\bz,V_0}^\bQ$, we conclude by Fact \ref{le:prealg} that $\ag{\bx',V_0}^\bQ\leq V^\A$.
\end{proof}

We now present the main lemma of this section, which allows us to extend $V_0^\A\to V_0^\B$ to a $K$-powered field isomorphism.

\begin{proposition}\label{prop:vsToField}
  Let $\A,\B$ be $K$-powered fields, let $V_0^\A,V_0^\B$ be strong $K$-subspaces of $V^\A,V^\B$ respectively, and let $\iota:V_0^\A\to V_0^\B$ be a $K$-linear isomorphism such that $\ex(V_0^\A)\equiv^{\rcf}\ex(V_0^\B)$. Assume $\A$ satisfies KR and is $|V^\B|^+$-saturated. If $\B=\ag{V_0^\B}^{\pc}_\B$ then there is $\B'\leq\A$ and a $K$-powered field isomorphism $f:\B'\to\B$ extending $\iota$.
\end{proposition}
\begin{proof}
  First, let us extend $\iota$ to include $\kex^\B$ in its image. Take a $\bQ$-basis $\by$ of $\kex^\B$ over $V_0^\B$. Strongness of $V_0^\B$ implies that $\by$ is $K$-independent over $V_0^\B$. Inductively applying Lemma \ref{le:extendKernel} and compactness, we conclude that there is $K$-independent $\bx\se\kex^\A$ such that $\ex\ag{\bx,V_0^\A}^K\equiv^{\rcf}\ex\ag{\by,V_0^\B}^K$. Thus, we may assume that $\kex^\B\se V_0^\B$.

  We construct inductively strong $K$-subspace $V_n$ of $V^\A$ and $f_n:V_n\to V^\B$ inducing isomorphisms $V_n\equiv^{\KVS}f(V_n)$ and $\ex(V_n)\equiv^{\rcf}\ex(f(V_n))$. Let $f_0=\iota$.

  Suppose we already have $V_n$ and $f_n$. Let $R_n:=\ag{\ex(V_n)}^{\rc}$ and $f_n':R_n\to R^\B$ extending $\ex(V_n)\to\ex(f_n(V_n))$. Take a multiplicative basis $(m_i)_{i\in I}$ of $\ag{\ex(V_n)}^S$ over $\ex(V_n)$. For each $i$, take $a_i\in\ex^{-1}(m_i)$ and $b_i\in\ex^{-1}(f_n'(m_i))$. Strongness implies that $(a_i)_{i\in I}$ is $K$-independent over $V_n$ and $(b_i)_{i\in I}$ is $K$-independent over $f(V_n)$. By Lemma \ref{le:addKernel}, we obtain $K$-independent $(a'_i)_{i\in I}=(a_i+z_i)_{i\in I}$ over $V_n$ such that $\ag{V_n,(a_i')_{i\in I}}^\bQ$ is strong and
  \[\ex\ag{(a_i')_{i\in I}}^K\equiv^{\rcf}_{\ex(V_n)}\ex\ag{(b_i)_{i\in I}}^K.\]
  Now let $V_{n+1}:=\ag{(a_i')_{i\in I},V_n}^K$ and $f_{n+1}$ extend $f_n$ mapping $(a_i')_{i\in I}$ to $(b_i)_{i\in I}$.

  Take $V'=\cup_n V_n,R'=\cup_n R_n,\B'=(V',R')$, and $f=\cup_nf_n$.
\end{proof}

\section{Strong pairs}\label{sec:str}
Let $V$ be a $K$-vector space, $R$ be a real closed field, and let $S$ be the unit circle of $R^2$. In this section, we introduce a pair of definable subsets of $V^n$ and $S^n$ to characterize a strong extension.

Let $D+\ba\se V^n$ be an affine $K$-linear subspace, and let $W_{\bb}\se S^n$ be a semialgebraically connected Nash manifold (semialgebraic analytic manifold) defined over $\bb$, which is given by an $\emptyset$-definable set $W(\bx;\by)\se S^n\times R^m$ and parameters $\bb\in R^m$.

\begin{fact}[{\cite[Proposition 8.4.1]{Bo98}}]\label{fact:Nash}
  Let $W\se R^n$ be a semialgebraically connected Nash manifold. Then the Zariski closure $W^{\zar}$ of $W$ in $R^n$ is irreducible, and $W^{\zar}$ has the same dimension as $W$.
\end{fact}

\begin{definition}
  We say
  \begin{itemize}
    \item $D+\ba$ is \emph{free} if $\bv(D)$ has dimension $1$ for every $\bv\in\bZ^n\setminus\{0\}$,
    \item $W_{\bb}$ is \emph{free} if $(W_{\bb})^{\bv}$ has dimension $1$ for every $\bv\in\bZ^n\setminus\{0\}$,
    \item $(D+\ba,W_{\bb})$ is \emph{free} if both $D+\ba$ and $W_{\bb}$ are free,
    \item $(D+\ba,W_{\bb})$ is \emph{rotund} if \[\dim M(D)+\dim(W_{\bb})^M\geq k\] for every $M\in\bZ^{k\times n}$ of rank $k$.
    \item $(D+\ba,W_{\bb})$ is a \emph{strong pair} if it is free and rotund.
  \end{itemize}
  The dimension of $(D+\ba,W_{\bb})$ is the sum $\dim D+\dim W_{\bb}$.
\end{definition}

We can alternatively define freeness and rotundity from the perspective of generic points.

\begin{lemma}\label{le:free}
  Let $\ba\se A\se V$ and $\bb\se B\se R$. Suppose $\bc$ is generic in $D+\ba$ over $A$ (or in $W_{\bb}$ over $B$). Then $\bc$ is $\bQ$-independent over $A$ (or multiplicatively independent over $B$) if and only if $D+\ba$ (or $W_{\bb}$) is free.
\end{lemma}
\begin{proof}
  $(\Rightarrow)$ is clear. We show $(\Leftarrow)$. Suppose $\bv\cdot\bc=\alpha\in A$ for some $\bv\in\bZ^n\setminus\{0\}$. Let $E=\ker(\bv:V^n\to V)$. Both $E+\bc$ and $D+\bc$ are definable over $A$, and hence,
  \[\dim(E+\bc)\cap(D+\bc)\geq\ld^K(\bc/A)=\dim(D+\bc).\]
  Therefore, $D\se E$. The proof for $W_{\bb}$ is similar.
\end{proof}

\begin{proposition}
  \label{prop:rotundStrong}
  Let $\X\se\Y\in\C_\I$ and let $(D+\ba,W_{\bb})$ be a free pair with $\ba\se V^\X$ and $\bb\se R^\X$. Let $\bc\se V^\Y$ satisfying $\bc$ is generic in $D+\ba$ over $V^\X$ and $\ex(\bc)$ is generic in $W_{\bb}$ over $R^\X$. Then $(D+\ba,W_{\bb})$ is rotund if and only if $V^\X\leq\ag{V^\X,\bc}^\bQ$.
\end{proposition}
\begin{proof}
  $(\Rightarrow)$ For every $\bx\se\ag{V^\X,\bc}^\bQ$, there is some $M\in\bZ^{k\times n}$ of rank $k$ satisfying
  \[V^\X+\ag{\bx}^\bQ=V^\X+\ag{M(\bc)}^\bQ.\]

  Freeness implies $\ld^\bQ(M(\bc)/V^\X)=k$, and rotundity implies
  \begin{align*}
    \delta(\bx/V^\X)&=\delta(M(\bc)/V^\X)\\
    &=\ld^K(M(\bc)/V^\X)+\trd(\ex(\bc)^M/R^\X)-\ld^\bQ(M(\bc)/V^\X)\\
    &=\dim M(D)+\dim(W_{\bb})^M-k\\
    &\geq 0
  \end{align*}

  $(\Leftarrow)$ For every $M$ of rank $k$, check that
  \[\ld^K(M(\bc)/V^\X)+\trd(\ex(\bc)^M/R^\X)\geq k.\qedhere\]
\end{proof}

\subsection{Definability of strongness}

In this subsection, we show that strongness is a definable property of the pair $(D+\ba,W_{\bb})$ when parameters vary. Since strongness does not depend on the parameter $\ba$, we may assume $\ba=0$. Let $\widehat{W_{\bb}}\se F^n$ denote the Zariski closure of $W_{\bb}\se S^n\se R^{2n}$ viewed as a subset of $F^n$.

\begin{lemma}
  Let $W_{\bb}$ be a semialgebraically connected Nash manifold. Then $\widehat{W_{\bb}}$ is irreducible, and the dimension of $W_{\bb}\se S^n$ equals the dimension of $\widehat{W_{\bb}}$ as a subvariety of $F^n$.
\end{lemma}
\begin{proof}
  Irreducibility follows from Fact \ref{fact:Nash}. For the dimension, take a generic point $(\bx,\by)\in W_{\bb}$. Then $\bx+i\by$ is a generic in $\widehat{W_{\bb}}$ and $\trd(\bx,\by/\bb)=\trd(\bx+i\by/\bb)$.
\end{proof}

Let $(X_{\bt})_{\bt\in F^m}\se F^n$ be a constructible family defined over $\bQ$. Applying Fact \ref{fact:WCIT} to $(X_{\bt})_{\bt\in F^m}$, we obtain a finite set $\mu(X)$ of proper algebraic subgroups of $(F^\times)^n$. For each $B\in\mu(X)$, we write $M_B$ for a full-rank matrix defining $B$. Let $\phi_{DWX}(\bs\bt)$ be the conjunction of the following formulas:
\begin{enumerate}
  \item $W_{\bs}\se X_{\bt}\;\bigwedge\;\dim W_{\bs}=\dim X_{\bt}$.
  \item $\bigvee_{B\in\mu(X)}\dim(X_{\bt})^{M_B}\geq 1$.
  \item $\bigvee_{B\in\mu(X)}\dim(X_{\bt})^{M_B}+\dim M_B(D)\geq\rk M_B$.
\end{enumerate}

\begin{lemma}
  Suppose $W_{\bb}$ is a semialgebraically connected analytic manifold, $X_{\bc}$ is an irreducible variety, and $(D,W_{\bb})$ is $n$-dimensional. Then $\phi_{DWX}(\bb\bc)$ holds if and only if all the followings hold:
  \begin{enumerate}
    \item $\widehat{W_{\bb}}=X_{\bc}$.
    \item $W_{\bb}$ is free.
    \item $(D+\ba,W_{\bb})$ is rotund.
  \end{enumerate}
\end{lemma}
\begin{proof}
  The equivalence of line 1 is clear. Let us assume $\widehat{W_{\bb}}=X_{\bc}$ and show that the remaining lines are equivalent. $(\Leftarrow)$ is clear by the definition and the fact that $\dim(W_{\bb})^{M_B}=\dim(\widehat{W_{\bb}})^{M_B}$.

  $(\Rightarrow)$ of line 2: Suppose on the contrary that there is $\bv\in\bZ^n\setminus\{0\}$ such that $W_{\bb}\se\bd\cdot\ker(\bv:F^n\to F)$ for some $\bd\in S^n$. Taking closure on both sides, we get $\widehat{W_{\bb}}\se\bd\cdot\ker(\bv)$, and therefore, $\widehat{W_{\bb}}\cap\bd\cdot\ker(\bv)$ is atypical. Fact \ref{fact:WCIT} implies that $(\widehat{W_{\bb}})^{M_B}$ is constant for some $B\in\mu(X)$.

  $(\Rightarrow)$ of line 3: Towards a contradiction, suppose $(D+\ba,W_{\bb})$ is not rotund. Then there is a full-rank matrix $M$ such that $\dim(W_{\bb})^M+\dim M(D)<\rk M<n$. Let $A\se(F^\times)^n$ be the algebraic subgroup defined by $\bz^M=\mathds{1}$.

  Let $(\bx,\by)$ be generic in $W_{\bb}$ over $\bb$. Then $\bz:=\bx+i\by$ is generic in $\widehat{W_{\bb}}$ over $\bb$. Let $S$ be an irreducible component of $\widehat{W_{\bb}}\cap\bz A$ containing $\bz$. The following calculation shows that $S$ is an atypical component of the intersection:
  \begin{align*}
    \dim S&=\dim(\widehat{W_{\bb}}\cap\bz A)\\
    &=\dim\widehat{W_{\bb}}-\dim(\widehat{W_{\bb}})^M&&\text{by fiber dimension}\\
    &>\dim\widehat{W_{\bb}}-\rk M\\
    &=\dim\widehat{W_{\bb}}+\dim\bz A-n.
  \end{align*}
  By Fact \ref{fact:WCIT}, there exists $B\in\mu(X)$ such that $S$ is contained in $\bz B$.

  For every coset of algebraic subgroup $\bz'A'\se\bz A$ such that $S\se\bz'A'$, we have
  \[\dim(\widehat{W_{\bb}}\cap\bz A)\geq\dim(\widehat{W_{\bb}}\cap\bz'A')\geq\dim S=\dim(\widehat{W_{\bb}}\cap\bz A).\]
  Therefore, we may assume $A$ to be the smallest algebraic subgroup having a coset containing $S$. Thus, $\bz A\se\bz B$. The typicality of $S$ with respect to $\bz B$ gives the following equality:
  \[\dim(\widehat{W_{\bb}}\cap\bz A)=\dim S=\dim(\widehat{W_{\bb}}\cap\bz B)+\dim A-\dim B,\]
  and by fiber dimension,
  \[\dim(\widehat{W_{\bb}})^{M_B}-\dim(\widehat{W_{\bb}})^M=\rk M_B-\rk M.\]
  Putting everything together, we get
  \begin{align*}
    \dim M(D)&<\rk M-\dim(\widehat{W_{\bb}})^M&&\text{by assumption}\\
    &=\rk M_B-\dim(\widehat{W_{\bb}})^{M_B}\\
    &\leq\dim M_B(D),&&\text{by }\phi_{DWX}(\bb\bc)
  \end{align*}
  contradicting $A\se B$.
\end{proof}

\begin{corollary}\label{cor:rotundDefinable}
  There is a set of formulas $\Phi_{DW}(\bs)$ such that for every $\bb\se R$ we have $\Phi_{DW}(\bb)$ holds if and only if and $(D,W_{\bb})$ is an $n$-dimensional strong pair.
\end{corollary}

\subsection{Existence of strong pairs}

\begin{lemma}
  There is an irreducible subvariety $W\se S^n$ such that $\dim W=n-1$ and $\dim W^M=k$ for every $M\in\bZ^{k\times n}$ of rank $k<n$.
\end{lemma}
\begin{proof}
  Let $W(x_1,y_1,...,x_n,y_n)$ be the subvariety of $S^n$ defined by
  \[x_1+...+x_n=n-1.\]
  Let us show that $\dim W^M=n-1$ for all $M\in\bZ^{(n-1)\times n}$ of rank $n-1$. It suffices to consider the $\bR$-points of $W$.

  By fiber dimension, we need to show that the intersection of $W$ and any coset of a $1$-dimensional subtorus has dimension 0. Suppose on the contrary that it intersects with a coset parametrized by
  \[\{(e^{i(v_1\theta+\alpha_1)},...,e^{i(v_n\theta+\alpha_n)})\mid\theta\in\bR\},\]
  where $(v_1,...,v_n)\in\bZ\setminus\{0\}$ and $\alpha_1,...,\alpha_n\in\bR$. If $v_1,...,v_k$ are non-zero, then
  \[f(\theta):=\cos(v_1\theta+\alpha_1)+...+\cos(v_k\theta+\alpha_k)\]
  must be a constant for all $\theta$ in some open interval. This is only possible when they cancel each other, which implies $f(\theta)=0$ and $k>1$. But then
  \[\cos\alpha_{k+1}+...+\cos\alpha_n\leq n-k<n-1,\]
  contradicting the definition of $W$.
\end{proof}

Since $K$ is a proper extension of $\bQ$, we can always find a $1$-dimensional free $K$-linear subspace $D\se V^n$.

\begin{corollary}\label{cor:existStrong}
  There is an $n$-dimensional strong pair $(D,W)\se V^n\times S^n$ such that $\dim D+\dim W=n$ and $\dim M(D)+\dim W^M>k$ for every $M\in\bZ^{k\times n}$ of rank $k<n$.
\end{corollary}

\section{Axiomatization}\label{sec:axi}

This section is devoted to a complete axiomatization of rich structures in $\C_\I$. We already see in Section \ref{sec:C_I} that $\Th(\C_\I)=T_\I$. The following property characterizes the $\leq$-existential closedness of rich structures.

\begin{definition}
  A structure $\A\in\C$ has the \emph{EC-property} if for every $n$-dimensional strong pair $(D+\ba,W_{\bb})\se V^n\times S^n$ and proper affine $K$-linear subspaces $D_1+\ba_1,...,D_k+\ba_k$ of $D+\ba$ where $\ba,\ba_1,...,\ba_k\se V^\A$ and $\bb\se R^\A$, the following set is non-empty:
  \[\big((D+\ba)\setminus\cup_{i=1}^k(D_i+\ba_i)\big)\cap\ex^{-1}(W_{\bb}).\]
\end{definition}

Corollary \ref{cor:rotundDefinable} enables us to express the EC-property in a set of first order sentences EC. Define the theory $T^{\rich}_\I$ as the union of $T_\I$, KR, and EC.

\begin{theorem}
  \label{thm:implyRich}
  Let $\U$ be an $\aleph_0^+$-saturated model of $T_\I^{\rich}$. Then $\U$ is a rich structure in $\C_\I$.
\end{theorem}
\begin{proof}
  Let $\X\leq\Y$ in $\Cf_\I$ be a minimal strong extension. Assume that $\X$ is strongly embedded into $\U$. To simplify the notation, let us further assume that $V^\X\se V^\U$ and $R^\X\se R^\U$. We need to show that $\Y$ is strongly embedded into $\U$. By Proposition \ref{prop:minimalExtension}, there are three cases:
  \begin{enumerate}[leftmargin=*]
    \item Kernel extension: $\ld^\bQ(\kex^\Y/V^\X)=1$ and $\Y=\ag{V^\X}^{\pc}_\Y$.
    
    Let $b\in\kex^\Y\setminus V^\X$. By Lemma \ref{le:extendKernel}, there is $a\in\kex^\U\setminus V^\X$ such that $\ex\ag{a}^K\equiv^{\rcf}_{R^\X}\ex\ag{b}^K$. Clearly, $V^\X\leq_0\ag{V^\X,a}^K\leq V^\U$. Applying Proposition \ref{prop:vsToField} to $\ag{V^\X,a}^K$, we obtain $\Y'\leq\U$ and $\Y'\cong\Y$ fixing $V^\X\se V_0$.
    
    \item Prealgebraic extension: $\kex^\X=\kex^\Y$ and $\Y=\ag{V^\X,\bb}^{\pc}_\Y$ for some $\bQ$-independent tuple $\bb=(b_1,...,b_n)\se V^\Y$ over $V^\X$ with $\delta(\bb/V^\X)=0$. \label{it:delta0}
    
    \begin{claim}
      For every $\bQ$-independent tuple $\bb'=(b_1,...,b_n,b_{n+1},...,b_m)\se\ag{\bb}^K$ over $V^\X$ extending $\bb$, there exists $\ba'\se V^\U$ such that $\ba'\equiv_{V^\X}^{\KVS}\bb'$ and $\ex(\ba')\equiv_{R^\X}^{\rcf}\ex(\bb')$.
    \end{claim}
    \begin{pfcl}
      By $\aleph_0^+$-saturation, it suffices to show that the type
      \[\tp^{\KVS}(\bb'/V^\X)\cup\tp^{\rcf}(\ex(\bb')/V^\X)\]
      is finitely satisfiable. Let $\mu$ be a finite subset of this type. By analytic cell decomposition, there exist affine $K$-linear subspaces $A,A_1,...,A_k\se V^n$ and an analytic $n$-cell $W\se S^n$ defined over $\X$, such that $\mu(\bx)$ can be deduced from
      \[\bx\in A\setminus\cup_{i}^k A_i\,\bigwedge\,\ex(\bx)\in W.\]
      By definition, analytic cells are connected analytic manifolds. Since $\X\leq_0\Y$, by Proposition \ref{prop:rotundStrong}, the pair $(A,W)$ is strong, and we can take $(A,W)$ to be $m$-dimensional because $\delta(\bb')=0$. The conclusion follows from the EC-property.\hfill$\blacksquare$
    \end{pfcl}

    By $\aleph_0^+$-saturation, there is $\ba\se V^\U$ such that $\ag{\ba}^K\equiv_{V^\X}^{\KVS}\ag{\bb}^K$ and $\ex\ag{\ba}^K\equiv_{R^\X}^{\rcf}\ex\ag{\bb}^K$. As a result, $V^\X\leq_0\ag{V^\X,\ba}^K\leq V^\U$. We conclude by Proposition \ref{prop:vsToField} applied to $\ag{V^\X,\ba}^K$.

    \item Purely transcendental extension: $\kex^\X=\kex^\Y$ and $\Y=\ag{V^\X,x}^{\pc}_\B$ for every $x\in V^\Y\setminus V^\X$.
    
    Take $b\in V^\Y\setminus V^\X$.
    \begin{claim}
      There exists $a\in V^\U$ such that
      \begin{itemize}
        \item $\ex(a)\equiv_{R^\X}^{\rcf}\ex(b)$, and
        \item $\forall\,\bw\in(V^\U)^n,\,\delta(\bw,a/V^\X)>0$.
      \end{itemize}
    \end{claim}
    \begin{pfcl}
      Since $V^\X\leq_0 V^\X\cup\kex^\U$, the second line is equivalent to
      \[\forall\,\bw\in(V^\U)^n,\,\delta(\bw,a/V^\X\cup\kex^\U)>0,\]
      which is type-definable by Proposition \ref{prop:ptType}. We prove finite satisfiability. Let $c_1,c_2\in\R^\X$ be such that $c_1<p_{\st}(\ex(b))<c_2$ and let $N\in\bN$. Consider the following partial type $\psi(t)$:
      \begin{itemize}
        \item $c_1<p_{\st}(\ex(t))<c_2$, for some $c_1,c_2\in R^\X$, and
        \item for all $n\leq N-2$, $\forall\,\bw\in V^n,\,\delta(\bw,a/V^\X)>0$.
      \end{itemize}

      By Corollary \ref{cor:existStrong}, there is an $N$-dimensional strong pair $D\times W$ with a generic point
      \[(\ba,\br)=(a_1,a_2,...,a_N,r_1,...,r_N)\in V^N\times S^N\] such that
      \begin{itemize}
        \item $c_1<p_{\st}(r_1)<c_2$,
        \item $\ld^K(\ba/V^\X)+\trd(\br/R^\X)=N$, and
        \item $\ld^K(M(\ba)/V^\X)+\trd(\br^M/R^\X)>k$, for every $M\in\bZ^{k\times N}$ of rank $k\in(0,N)$.
      \end{itemize}
      Applying Proposition \ref{prop:extendToStructure} to $V_0=\ag{V^\X,\ba}^\bQ,R_0=\ag{R^\X,\br}^{\rc}$ and $\ex_0$ mapping $\ba$ to $\br$, we obtain a structure $\X'\in\Cf_\I$ such that $\ag{V^\X,\ba}^\bQ\leq V^{\X'}$, $\kex^{\X'}=\kex^\X$, and $\ag{V^\X,\ba}^{\pc}_{\X'}=\X'$. By Proposition \ref{prop:rotundStrong}, we have $\X\leq\X'\in\Cf_\I$. Since $\delta(\ba/V^\X)=N$, case \eqref{it:delta0} implies that $\X'$ is strongly embedded into $\U$. Identify $\X'$ with its image in $\U$.

      Let us show that $\U\vDash\psi(a_1)$. For the second line, take any $\bw\in(V^\U)^n$, where $n\leq N-2$. Take a $\bQ$-basis $\bw'$ of $\bw\cap\ag{V^\X,\ba}^\bQ$ over $V^\X$. It follows from submodularity and $\ag{V^\X,\ba}\leq V^\U$ that
      \[\delta_{V^\X}(\bw/\bw',a_1)\geq\delta_{V^\X}(\bw/\ba)\geq 0.\]
      Take a matrix $M\in\bZ^{k\times N}$ of rank $k\in(0,N)$ such that $\ag{V^\X,\bw',a_1}^\bQ=\ag{V^\X,M(\ba)}^\bQ$. We have
      \begin{align*}
        \ld^K(\bw',a_1/V^\X)+\trd(\ex(\bw'),r_1/R^\X)=\ld^K(M(\ba)/V^\X)+\trd(\br^M/R^\X)>k,
      \end{align*}
      i.e., $\delta_{V^\X}(\bw',a_1)>0$. Combining the two inequalities, we get
      \[\delta_{V^\X}(\bw,a_1)=\delta_{V^\X}(\bw/\bw',a_1)+\delta_{V^\X}(\bw',a_1)>0.\tag*{$\blacksquare$}\]
    \end{pfcl}

    Take such an $a\in V^\U$. The following inequality shows that the set $\ag{V^\X,a}^\bQ$ is strong in $V^\U$:
    \[\delta(\bw/V^\X,a)=\delta(\bw,a/V^\X)-\delta(a/V^\X)=\delta(\bw,a/V^\X)-1\geq 0.\]
    Apply Lemma \ref{le:addKernel} to $a\in V^\U$, $b\in V^\B$ to obtain $a'\in V^\U$ such that $\ex\ag{a'}^K\equiv^{\rcf}_{R^\X}\ex\ag{b}^K$, and then conclude by Proposition \ref{prop:vsToField}. \qedhere
  \end{enumerate}
\end{proof}

\begin{corollary}[Theorem \ref{thm:axiomatization}]
  The theory of rich structures in $\C_\I$ is axiomatized by $T_\I^{\rich}$.
\end{corollary}

\section{Open core}\label{sec:ope}

In this section, we show that every open subset of $R^n$ definable in $T_\I^{\rich}$ must be semialgebraic. We use a criterion by Boxall and Hieronymi \cite{Bo12} for an expansion of a topological structure $\M$ to preserve the open core.

\begin{fact}[{\cite[Corollary 3.1]{Bo12}}]\label{le:openCore}
  Let $\M^*$ be an expansion of $\M$ by language, and let $C$ be a set of parameters in $\M$. Assume $\M^*$ is sufficiently saturated and strongly homogeneous. Suppose that for all sorts $S_1,...,S_n$ there is a set $D_{S_1...S_n}\se S_1^\M\times...\times S_n^\M$ such that the following conditions hold:
  \begin{enumerate}
    \item $D_{S_1...S_n}$ is dense in $S_1^\M\times...\times S_n^\M$.
    \item For every $\ba\in D_{S_1...S_n}$ and every open $U\se S_1^\M\times...\times S_n^\M$, if $\tp^\M(\ba/C)$ is realized in $U$ then $\tp^\M(\ba/C)$ is realized in $U\cap D_{S_1...S_n}$.\label{it:open2}
    \item For every $\bx\in D_{S_1...S_n}$, the conjunction of $\tp^\M(\bx/C)$ and $\bx\in D_{S_1...S_n}$ implies $\tp^{\M^*}(\bx/C)$.\label{it:open3}
  \end{enumerate}
  Then every open set definable over $C$ in $\M^*$ is definable over $C$ in $\M$.
\end{fact}

Let $\M^*=(V,R,\ex)$ be a saturated model of $T_\I^{\rich}$ where $V$ is equipped with the trivial topology, and let $\M$ be the reduct $(V,R)$. Let $\bc\se V$ be a tuple of parameters. We may assume $\ag{\bc}^\bQ\leq V$ by extending $\bc$ via Fact \ref{le:min}. Applying Lemma \ref{le:minimalStrong}, we obtain a $K$-powered subfield $\A\in\Cf$ strong in $\M^*$ such that $\A=\ag{\bc}^{\pc}_\A$. We want to show that every open set definable over $\A$ in $\M^*$ is definable over $\A$ in $\M$.

Let $D_n'$ be the following set:
\[\{(a_1,...,a_n)\mid(a_1,...,a_n)\in V^n\text{ such that }\forall\,\bw\in V^m,\,\delta(\bw\ba/V^\A)\geq n\}.\]
In other words, $D_n'$ consists of all $\Cl$-independent tuples over $V^\A$. Let $D_n\se R^n$ be the set $(p_{\st}\circ\ex)^n(D_n')$.

\begin{lemma}\label{le:open1}
  Let $U$ be an open rectangle in $R^n$. Then $U\cap D_n\neq\emptyset$.
\end{lemma}
\begin{proof}
  We show that for every open arcs $U_1,...,U_n\se S$ defined over $\bd$, the set $(U_1\times...\times U_n)\cap\ex(D_n')$ is non-empty. By extending $\bd$, we may assume $\ag{V^\A,\bd}^\bQ\leq V$. Lemma \ref{le:minimalStrong} yields a $K$-powered subfield $\B\in\Cf$ strong in $\M^*$ such that $\B=\ag{V^\A,\bd}^{\pc}_\B$. By Proposition \ref{prop:extendToStructure}, there exists a structure $\B_1\in\Cf$ extending $\B$ and $a_1,...,a_n\in V^{\B_1}$ such that:
  \begin{itemize}
    \item $\ag{V^\B,\ba}^\bQ\leq V^{\B_1}$,
    \item $\kex^{\B_1}=\kex^\B$,
    \item $\B_1=\ag{V^\B,\ba}^{\pc}_{\B_1}$,
    \item $\ld^K(\ba/V^\B)=\trd(\ex(\ba)/R^\B)=n$, and
    \item $\ex(a_i)\in U_i$ for $i=1,...,n$.
  \end{itemize}
  Richness of $\M^*$ implies there is a strong embedding $\B_1\to\M^*$ extending $\B\leq\M^*$. Let us identify $\B_1$ with its image in $\M^*$. For every $\bw\in V^m$, strongness gives
  \[\delta(\bw\ba/V^\B)=\delta(\bw/\ba\cup V^\B)+\delta(\ba/V^\B)\geq n.\]
  Hence, $\ba$ is $\Cl$-independent over $V^\B$. We conclude that $\ba\in(U_1\times...\times U_n)\cap\ex(D_n')$.
\end{proof}

\begin{lemma}\label{le:open3}
  Let $\ba,\bb\in D_n$ satisfying $\ba\equiv^{\rcf}_{R^\A}\bb$. Then there is a $K$-powered field isomorphism
  \[f:\X\xrightarrow{\sim}\Y,\text{ where }\X,\Y\leq\M^*\text{ and }\X,\Y\in\Cf,\]
  sending $\ba$ to $\bb$ and fixing $\A$.
\end{lemma}
\begin{proof}
  Let $\bx,\by\in V$ such that $p^n(\ex(\bx))=\ba$ and $p^n(\ex(\by))=\bb$. We show there is such an isomorphism sending $\ex(\bx)$ to $\ex(\by)$. First, notice that $D_n'$ implies both $\ag{\bx,V^\A}^\bQ$ and $\ag{\by,V^\A}^\bQ$ are strong in $V$, both $\bx,\by$ are $K$-independent tuples over $V^\A$, and both $\ex(\bx),\ex(\by)$ are algebraically independent tuples over $R^\A$. Therefore, $\bx\equiv^{\KVS}_{V^\A}\by$ and $\ex(\bx)\equiv^{\rcf}_{R^\A}\ex(\by)$. Now apply Lemma \ref{le:addKernel} to $\ag{V^\A,\bx}^\bQ$ and $\ag{V^\A,\by}^\bQ$, and then conclude by Proposition \ref{prop:vsToField}.
\end{proof}

\begin{proposition}
  Every open set definable over $\A$ in $\M^*$ is definable over $\A$ in $\M$.
\end{proposition}
\begin{proof}
  Apply Fact \ref{le:openCore} with $C=V^\A\cup R^\A$. The first and third lines are Lemma \ref{le:open1} and \ref{le:open3} respectively. The second line follows from the fact that $\ba\in D_n$ implies $\ba$ is algebraically independent over $R^\A$.
\end{proof}

\begin{corollary}
  Every open subset of $R^n$ definable in $\M^*$ is semialgebraic.
\end{corollary}

\section{Model}\label{sec:mod}

Recall $\bR^K=(\bR_{\KVS},\bR_{\oring},\exp)$, where $\exp(t)=e^{2\pi it}$.

In this section, we show $\bR^K$ is a model of $T_\I^{\rich}$ for some $\I$ under the assumption SC$_K$ (Conjecture \ref{con:SCK}). In fact, we only need to assume $\delta$ has a lower bound.

\begin{lemma}
  Assume there is $d\in\bZ$ such that for all $\bx\se\bR$ we have
  \[\delta(\bx)=\ld^K(\bx)+\trd(e^{2\pi i\bx})-\ld^\bQ(\bx)\geq d.\]
  Then $\bR^K$ contains a finitely generated $K$-powered subfield $\I$ such that $\I\leq\bR^K$.
\end{lemma}
\begin{proof}
  Apply Fact \ref{le:min} to the empty set and then apply Lemma \ref{le:minimalStrong}.
\end{proof}

As a result, $\bR^K\vDash T_\I$. The property KR is clear for $\bR^K$ by Fact \ref{fact:Kronecker}. We prove EC for $\bR^K$.

Let $(D+\ba,W_{\bb})\se\bR^n\times(S^1)^n$ be an $n$-dimensional strong pair and let $D_1+\ba_1,...,D_k+\ba_k$ be proper affine $K$-linear subspaces of $D+\ba$. We may assume $\ba=0$ by replacing $W_{\bb}$ with $\exp(-\ba)\cdot W_{\bb}$, and each $D_i$ has dimension $\dim D-1$ by taking a larger $D_i$. We need to show that $D\setminus\cup_{i=1}^k(D_i+\ba_i)\cap\exp^{-1}(W_{\bb})$ is non-empty.

\begin{lemma}
  There is a semigroup $D'\se D$ open in $D$ such that $D'$ does not intersect with any $D_i+\ba_i$ for $i=1,...,k$.
\end{lemma}
\begin{proof}
  We show that there is an open cone $C\se D$ and a closed ball $B_r\se\bR^n$ of radius $r$ centered at origin such that $C\setminus B_r$ satisfies the requirement. Each $D_i+\ba_i$ cut $D$ into two open subsets, and we let $E_i$ to be the one not containing $0$. Construct polygon $P_i$'s as follows:
  \begin{itemize}
    \item $P_0:=D$.
    \item If $D_i+\ba_i\cap P_{i-1}=\emptyset$ then $P_i:=P_{i-1}$.
    \item If $D_i+\ba_i\cap P_{i-1}\neq\emptyset$ then $P_i:=P_{i-1}\cap E_i$.
  \end{itemize}
  Note that $P_n$ is an unbounded polygon open in $D$. Hence, $P_n\cup B_r$ contains a cone open in $D$ for large enough $r\in\bR$.
\end{proof}

Take an open semigroup $D'\se D$ satisfying the lemma.

\begin{lemma}
  If $D$ is free, then $\exp(D')$ is dense in $(S^1)^n$.
\end{lemma}
\begin{proof}
  Since $D'$ is open, it contains a generic point $\ba$ of $D$ over $\bQ$. By Lemma \ref{le:free}, the set $\ba$ is $\bQ$-linearly independent over $\bQ$. By Fact \ref{fact:Kronecker}, the set $\exp(\ba\bN_{>0})$ is dense in $(S^1)^n$, and it is clear that $\ba\bN_{>0}\se D'$.
\end{proof}

Define the following map:
\begin{align*}
  f:D'\times W&\longrightarrow\;(S^1)^n\\
  (\bx,\by)\;&\longmapsto\exp(-\bx)\cdot\by.
\end{align*}
Let $D^*\times W^*$ be an open subset of $D'\times W$ such that $f|_{D^*\times W^*}$ is definable in $\bR_{\operatorname{an}}$.

\begin{lemma}
  For every generic $(\bx^*,\by^*)\in D^*\times W^*$, the fiber $f^{-1}(f(\bx^*,\by^*))$ has dimension $0$.
\end{lemma}
\begin{proof}
  Let $\bz^*=f(\bx^*,\by^*)$. Then
  \[f^{-1}(f(\bx^*,\by^*))=\{(\bx,\by)\in D^*\times W^*\mid\by=\bz^*\exp(\bx)\}.\]
  Suppose the contrary. Since $(\bx^*,\by^*)$ is generic, $f^{-1}(f(\bx^*,\by^*))$ contains an analytic curve through $(\bx^*,\by^*)$ parametrized by $(\tilde{\bx},\tilde{\by})(t):(-1,1)\to\bR^n\times(S^1)^n$ with $\tilde{\bx}(0)=\bx^*$ and $\tilde{\by}(t)=\bz^*\exp(\tilde{\bx})$. Clearly, $\tilde{x}_1(t),...,\tilde{x}_n(t)\in\bR[[t]]$ and $\rk J(\tilde{x}_1(t),...,\tilde{x}_n(t))>0$.

  Let $\hat{\bx}=\tilde{\bx}-\bx^*$. After swapping the coordinates, we may assume that the first $k$-terms $\hat{x}_1,...,\hat{x}_k$ form a $\bQ$-linear basis of $\hat{\bx}$, i.e., for each $j\leq n-k$ we have that $\hat{x}_{k+j}=a_{1j}\hat{x}_1+...+a_{kj}\hat{x}_k$ for some $a_{ij}\in\bQ$. Hence, there is a full-rank matrix $M\in\bZ^{(n-k)\times n}$ such that $\hat{\bx}\in\ker(M)$. By fiber dimension,
  \begin{align*}
    &\dim(D^*\cap(\bx^*+\ker M))=\dim D^*-\dim M(D^*)=\dim D-\dim M(D),\\
    &\dim(W^*\cap\by^*\ker((\cdot)^M))=\dim W^*-\dim(W^*)^M=\dim W-\dim(W)^M.
  \end{align*}
  Rotundity of $D\times W$ implies
  \begin{align*}
    \trd(\tilde{\bx}/\bC)+\trd(\tilde{\by}/\bC)&\leq\dim(D^*\cap(\bx^*+\ker M))+\dim(W^*\cap\by^*\ker((\cdot)^M))\\
    &=(\dim D+\dim W)-(\dim M(D)+\dim(W)^M)\\
    &\leq k.
  \end{align*}
  However, Ax-Schanuel (Fact \ref{fact:AxSch}) implies
  \[\trd(\tilde{\bx}\tilde{\by}/\bC)=\trd\big(\hat{x}_1,...,\hat{x}_k,\exp(\hat{x}_1),...,\exp(\hat{x}_k)/\bC\big)>k.\]
  Contradiction.
\end{proof}

\begin{lemma}
  $D'\cap\exp^{-1}(W)$ is non-empty.
\end{lemma}
\begin{proof}
  Let $X$ denote $D^*\times W^*$. The set
  \[X_0:=\{(\bx,\by)\in X|\dim f^{-1}(f(\bx,\by))=0\},\]
  is definable. By the previous lemma, $\dim(X\setminus X_0)<\dim X$. By fiber dimension,
  \[\dim f(X)\geq\dim f(X_0)=\dim X_0-0=\dim X=n.\]
  Since $\exp(D')$ is dense in $(S^1)^n$, there is $\alpha\in f(D^*\times W^*)\cap\exp(D')$, i.e., there are $\bu,\bv\in D',\bw\in W$ such that
  \[\exp(-\bu)\cdot\bw=\alpha=\exp(\bv).\]
  Hence,
  \[\exp(\bu+\bv)=\bw,\]
  where $\bu+\bv\in D'$.
\end{proof}

\begin{proposition}
  The structure $\bR^K$ satisfies EC.
\end{proposition}

We conclude that if $\bR^K\vDash T_\I$ then $\bR^K\vDash T_\I^{\rich}$.

\bibliography{power}

@article{Kr84,
 author = {Kronecker, Leopold},
 title = {N{\"a}herungsweise ganzzahlige {Aufl{\"o}sung} linearer {Gleichungen}.},
 journal = {Sitzungsberichte der K{\"o}niglich Preussischen Akademie der Wissenschaften},
 volume = {1884},
 pages = {1179--1193},
 year = {1884},
 language = {German},
 JFM = {16.0083.02}
}

@article{Na90,
 author = {Nadel, Mark and Stavi, Jonathan},
 title = {On models of the elementary theory of {{\((\mathbb{Z}{},+,1)\)}}},
 fjournal = {The Journal of Symbolic Logic},
 journal = {J. Symb. Log.},
 issn = {0022-4812},
 volume = {55},
 number = {1},
 pages = {1--20},
 year = {1990},
 language = {English},
 doi = {10.2307/2274950},
 zbMATH = {5700},
 Zbl = {0739.03023}
}

@article{Ki09,
 author = {Kirby, Jonathan},
 title = {The theory of the exponential differential equations of semiabelian varieties},
 journal = {Selecta Mathematica, New Series},
 issn = {1022-1824},
 volume = {15},
 number = {3},
 pages = {445--486},
 year = {2009},
 language = {English},
 doi = {10.1007/s00029-009-0001-7},
 zbMATH = {5640212},
 Zbl = {1263.12003}
}

@article{Ki10,
    author = {Bays, Martin and Kirby, Jonathan and Wilkie, A. J.},
    title = "{A Schanuel property for exponentially transcendental powers}",
    journal = {Bulletin of the London Mathematical Society},
    volume = {42},
    number = {5},
    pages = {917-922},
    year = {2010},
    month = {08},
    issn = {0024-6093},
    doi = {10.1112/blms/bdq054},
}

@article{La83,
 ISSN = {09895558},
 URL = {http://www.jstor.org/stable/44165476},
 author = {Michel Laurent},
 journal = {Séminaire de Théorie des Nombres de Bordeaux},
 pages = {1--8},
 publisher = {Société Arithmétique de Bordeaux},
 title = {EQUATIONS DIOPHANTIENNES EXPONENTIELLES},
 urldate = {2024-12-15},
 year = {1983}
}

@misc{Zi15,
      title={The theory of exponential sums}, 
      author={Boris Zilber},
      year={2015},
      HowPublished = {Preprint, {arXiv}:1501.03297 [math.{LO}]},
      url={https://arxiv.org/abs/1501.03297},
}

@article{Zi02,
author = {Zilber, Boris},
year = {2002},
pages = {27-44},
title = {Exponential sums equations and the {S}chanuel conjecture},
volume = {65},
number ={2},
journal = {Journal of The London Mathematical Society},
doi = {10.1112/S0024610701002861}
}

@article{Zi03,
 author = {Zilber, Boris},
 title = {Raising to powers in algebraically closed fields},
 journal = {Journal of Mathematical Logic},
 issn = {0219-0613},
 volume = {3},
 number = {2},
 pages = {217--238},
 year = {2003},
 language = {English},
 doi = {10.1142/S0219061303000273}
}

@article{Mi11,
 ISSN = {00029939, 10886826},
 URL = {http://www.jstor.org/stable/41059220},
 author = {Miller, Chris},
 journal = {Proceedings of the American Mathematical Society},
 number = {1},
 pages = {319--330},
 publisher = {American Mathematical Society},
 title = {EXPANSIONS OF O-MINIMAL STRUCTURES ON THE REAL FIELD BY TRAJECTORIES OF LINEAR VECTOR FIELDS},
 urldate = {2023-03-16},
 volume = {139},
 year = {2011}
}

@book{Dr98, place={Cambridge}, series={London Mathematical Society Lecture Note Series}, title={Tame Topology and O-minimal Structures}, DOI={10.1017/CBO9780511525919}, publisher={Cambridge University Press}, author={Dries, Lou van den}, year={1998}, collection={London Mathematical Society Lecture Note Series}}

@book{Bo06, place={Cambridge}, series={New Mathematical Monographs}, title={Heights in Diophantine Geometry}, DOI={10.1017/CBO9780511542879}, publisher={Cambridge University Press}, author={Bombieri, Enrico and Gubler, Walter}, year={2006}, collection={New Mathematical Monographs}}

@article{Ax71,
 ISSN = {0003486X},
 URL = {http://www.jstor.org/stable/1970774},
 author = {James Ax},
 journal = {Annals of Mathematics},
 number = {2},
 pages = {252--268},
 publisher = {Annals of Mathematics},
 title = {On {S}chanuel's {C}onjectures},
 urldate = {2023-03-31},
 volume = {93},
 year = {1971}
}

@book{Ma02,
  year = {2002},
  publisher = {Springer New York},
  author = {David Marker},
  title = {Model Theory : An Introduction}
}

@book{Bo98,
  title={Real Algebraic Geometry},
  author={Jacek Bochnak and Michel Coste and Marie-Françoise Roy},
  series={Ergebnisse der Mathematik und ihrer Grenzgebiete. 3. Folge / A Series of Modern Surveys in Mathematics},
  year={1998},
  publisher={Springer Berlin Heidelberg}
}

@article{Bo12,
 ISSN = {00224812},
 author = {Gareth Boxall and Philipp Hieronymi},
 journal = {The Journal of Symbolic Logic},
 number = {1},
 pages = {111--121},
 publisher = {Association for Symbolic Logic},
 title = {EXPANSIONS WHICH INTRODUCE NO NEW OPEN SETS},
 urldate = {2023-04-18},
 volume = {77},
 year = {2012}
}

@article{Pi88,
title = {On groups and fields definable in o-minimal structures},
journal = {Journal of Pure and Applied Algebra},
volume = {53},
number = {3},
pages = {239-255},
year = {1988},
issn = {0022-4049},
doi = {https://doi.org/10.1016/0022-4049(88)90125-9},
url = {https://www.sciencedirect.com/science/article/pii/0022404988901259},
author = {Anand Pillay}
}
\bibliographystyle{plain}

\end{document}